\documentclass[final,leqno]{siamltex704}
\usepackage{epsfig}
\usepackage{amsmath}
\usepackage{amssymb}
\usepackage{tikz}
\usepackage{graphicx}
\usepackage[notcite,notref]{showkeys}

\newtheorem{algorithm}{Algorithm}
\newcommand{\bq}{\begin{equation}}
\newcommand{\eq}{\end{equation}}

\newcommand{\be}{{\bf e}}
\newcommand{\bu}{{\bf u}}
\newcommand{\bw}{{\bf w}}
\newcommand{\bx}{{\bf x}}

\newcommand{\bt}{{\bf t}}

\newcommand{\bv}{{\bf v}}
\newcommand{\bQ}{{\bf Q}}

\def\T{{\mathcal T}}
\def\E{{\mathcal E}}
\def\pT{{\partial T}}

\def\bbf{{\bf f}}

\def\bn{{\bf n}}
\def\bq{{\bf q}}

\def\l{{\langle}}
\def\r{{\rangle}}

\def\3bar{{|\hspace{-.02in}|\hspace{-.02in}|}}

\title{A simple finite element method for the Stokes equations}

\author{ Lin Mu \thanks{Computer Science and Mathematics Division
Oak Ridge National Laboratory, Oak Ridge, TN, 37831,USA
(mul1@ornl.gov). The first author's research is based upon work supported in part by the U.S. Department of Energy, Office of Science, Office of Advanced Scientific Computing Research, Applied Mathematics program under award number ERKJE45; and by the Laboratory Directed Research and Development program at the Oak Ridge National Laboratory, which is operated by UT-Battelle, LLC., for the U.S. Department of Energy under Contract DE-AC05-00OR22725} \and Xiu Ye\thanks{Department of
Mathematics, University of Arkansas at Little Rock, Little Rock, AR
72204 (xxye@ualr.edu). This research was supported in part by
National Science Foundation Grant DMS-1620016.}}

 \graphicspath{{./picture/}} 
 
\begin{document}

\maketitle

\begin{abstract}
The goal of this paper is to introduce a simple finite element method 
to  solve the Stokes and the Navier-Stokes equations.  This method is in primal velocity-pressure formulation and is so simple such that both velocity and pressure are approximated by piecewise constant functions. Implementation issues as well as error analysis are investigated. A basis for a divergence free subspace of the velocity field is constructed so that the original saddle point problem can be reduced to a symmetric and positive definite system with much fewer unknowns. The numerical experiments indicate that the method is accurate and robust.
\end{abstract}



\pagestyle{myheadings}

\section{Introduction}

The Stokes problem is to
seek a pair of unknown functions $(\bu; p)$ satisfying
\begin{eqnarray}
-\nu\Delta\bu+\nabla p &=&\bbf\quad \mbox{in}\;\Omega,\label{moment}\\
\nabla\cdot\bu&=&0\quad\mbox{in}\;\Omega,\label{cont}\\
\bu &=& {\bf 0}\quad\mbox{on}\;\partial\Omega,\label{bc}
\end{eqnarray}
where $\nu$ denotes the fluid viscosity;  $\Delta$, $\nabla$, and
$\nabla\cdot$ denote the Laplacian, gradient, and divergence
operators, respectively; $\Omega \subset \mathbb{R}^d$ is the
region occupied by the fluid; $\bbf=\bbf(\bx)\in ([L^2(\Omega))]^d$
is the unit external volumetric force acting on the fluid at
$\bx\in\Omega$. For simplicity, we let $\nu=1$.

The weak formulation of the Stokes equations seeks $\bu\in [H_0^1(\Omega)]^d$ and $p\in L_0^2(\Omega)$ satisfying
\begin{eqnarray}
(\nabla\bu,\;\nabla\bv)-(\nabla\cdot\bv,\;p)&=&(\bbf,\;\bv),\quad \bv\in [H_0^1(\Omega)]^d\label{wf-m} \\
(\nabla\cdot\bu,\;q)&=&0,\quad\quad\quad q\in L_0^2(\Omega).\label{wf-c}
\end{eqnarray}

The linear Stokes equations are the limiting case of zero Reynolds number for the Navier-Stokes equations. The Stokes equations have attracted a substantial attention from researchers because of  its close relation with the Navier-Stokes equations. Numerical solutions of the Stokes equations have been investigated intensively and many different numerical schemes have been developed such as conforming/noconforming finite element methods, MAC method and finite volume methods. It is impossible to cite all the references. Therefore we just cite some classic ones \cite{cr,gr,gun,ht}.

In this paper, we present a finite element scheme for the Stokes equations  and its equivalent divergence free formulation. In this method, velocity is approximated by weak Galerkin element of degree $k=0$ and pressure is approximated by  piecewise polynomials of degree $k=0$. Weak Galerkin  refers to a general finite element technique for
partial differential equations in which differential operators are approximated by weak forms as distributions for generalized functions. The weak Galerkin finite element method first introduced in \cite{wy, wy-mixed} is a natural extension of the standard Galerkin finite element method for functions with discontinuity. 

One of the main difficulties in solving the the Stokes and the Navier-Stokes equations
is that the velocity and the pressure variables are coupled in a saddle point system. Many methods are developed to overcome
this difficulty. Divergence free finite element methods are such methods by approximating  velocity from weakly or exactly divergence free subspaces. As a consequence, pressure is eliminated from a saddle point system, along with the incompressibility
constraint resulting in a symmetric and positive definite system with a significantly
smaller number of unknowns.  For this simple finite element formulation, a divergence free basis is constructed explicitly. 

The rest of the paper is organized as follows.
The finite element formulation of this weak Galerkin method is introduced in Section \ref{section-wg}. Implementation issues of the method are discussed in Section \ref{section-implementation}.  In Section \ref{section-error}, we prove optimal order convergence rate of the method. Divergence free basis functions are constructed in Section \ref{section-divfree}. Using these basis functions, we can derive a divergence free weak Galerkin finite element formulation that will reduce a saddle point problem to a symmetric and positive definite system.  Numerical examples are presented in Section \ref{section-ne} to demonstrate the robustness and accuracy of the method.

\section{Finite Element Scheme}\label{section-wg}

Let ${\cal T}_h$ be a shape-regular triangulation of the domain
$\Omega$ with mesh size $h$.  Denote by ${\cal E}_h$
the set of all edges or  faces in ${\cal T}_h$, and let ${\cal
E}_h^0={\cal E}_h\backslash\partial\Omega$ be the set of all
interior edges or  faces.  let  ${\cal V}_h$ be the set of all interior vertices in $\T_h$. Define $N_E=card (\E_h^0)$, $N_V=card ({\cal V}_h)$ and $N_T=card (\T_h)$. For every element $T\in \T_h$, we
denote by $h_T$ its diameter and mesh size $h=\max_{T\in\T_h} h_T$
for ${\cal T}_h$.

The weak Galerkin methods create a new way to define a function $v$ that allows $v$ taking different forms in the interior and on the boundary of the element:
$$
v=
\left\{
\begin{array}{l}
  \displaystyle
  v_0,\quad {\rm in}\; T^0
  \\ [0.08in]
  \displaystyle
  v_b,\quad {\rm on}\;\partial T
 \end{array}
\right.
$$
where $T_0$ denotes the interior of $T$.
Since  weak function $v$ is formed by two parts $v_0$ and $v_b$, we write $v$ as $v=\{v_0,v_b\}$ in short without confusion. Let $P_k(T)$ denote the set consisting all the polynomials of degree less or equal to $k$.

Associated with $\T_h$, we define finite element spaces $V_h$ for velocity
\begin{equation}\label{vhspace}
V_h=\{\bv=\{\bv_0,\bv_b\}:\; \bv_0|_T\in [P_0(T)]^d,\ \bv_b|_e\in [P_0(e)]^d,\ e\in\pT,  T\in \T_h\}
\end{equation}
and $W_h$ for pressure
\begin{equation}\label{phspace}
W_h=\{q\in L^2_0(\Omega):\; q|_T\in P_0(T),\; T\in\T_h\},
\end{equation}
where $L^2_0(\Omega)$ is the subspace of $L_2(\Omega)$ consisting of functions with mean value zero.

We define $V_h^0$ a subspace of $V_h$ as
\begin{equation}\label{vh0space}
V^0_h=\{\bv=\{\bv_0,\bv_b\}\in V_h:\  \bv_b=0 \mbox{ on } \partial\Omega\}.
\end{equation}
We would like to emphasize that any function $v\in V_h$ has a single
value $v_b$ on each edge $e\in\E_h$.

Since the functions in $V_h$ are discontinuous polynomials, gradient operator $\nabla$ and divergence operator $\nabla\cdot$  in (\ref{wf-m})-(\ref{wf-c}) cannot be applied to them. Therefore we defined weak gradient and weak divergence for the functions in $V_h$. Let $RT_0(T)=[P_0(T)]^d+{\bf x}P_0(T)$ introduced in \cite{rt}. Let $\bn$ denote the unit outward normal.

For $\bv\in V_h$ and $T\in\T_h$, we define weak gradient
$\nabla_{w}\bv \in [RT_0(T)]^d$  as the unique polynomial
 satisfying the following
equation
\begin{equation}\label{dwg}
(\nabla_{w}\bv, \tau)_T = -(\bv_0,\nabla\cdot \tau)_T+ \langle \bv_b,
\tau\cdot\bn\rangle_{\partial T},\qquad \forall \tau\in [RT_0(T)]^d,
\end{equation}
and define weak divergence $\nabla_{w}\cdot\bv \in P_0(T)$ as the unique polynomial
 satisfying 
\begin{equation}\label{dwd}
(\nabla_{w}\cdot\bv, q)_T = -(\bv_0,\nabla q)_T+ \langle \bv_b,
q\bn\rangle_{\partial T},\qquad \forall q\in P_0(T).
\end{equation}

Define two bilinear forms as
\begin{eqnarray*}
a(\bv,\bw)=\sum_{T\in\T_h}(\nabla_w\bv,\nabla_w\bw)_T, \quad b(\bv,q)=\sum_{T\in\T_h}(\nabla_w\cdot\bv,q)_T.
\end{eqnarray*}

For each element $T\in \T_h$, denote by $Q_0$ and $\bQ_0$ the $L^2$ projections from $L^2(T)$ to $P_0(T)$ and
from $[L^2(T)]^d$ to $[P_0(T)]^d$ respectively. Denote by $\bQ_b$ the $L^2$ projection from
$[L^2(e)]^d$ to $[P_{0}(e)]^d$.

\medskip

\begin{algorithm}
A weak Galerkin method for the Stokes equations seeks
$\bu_h=(\bu_0,\bu_b)\in V_h^0$ and $p_h\in W_h$ satisfying  the following equation:
\begin{eqnarray}
a(\bu_h,\bv)-b(\bv,p_h)&=&(\bbf,\;\bv_0), \quad\forall \bv=\{\bv_0, \bv_b\}\in V_h^0\label{wg-m}\\
b(\bu_h,q)&=&0,\quad\quad\quad\forall q\in W_h.\label{wg-c}
\end{eqnarray}
\end{algorithm}

\section{Implementation of the method}\label{section-implementation}

The linear system associated with the algorithm (\ref{wg-m})-(\ref{wg-c}) is a saddle point problem with the form,
\begin{equation}\label{matrix}
\left(
  \begin{array}{cc}
    A & -B \\
    B^T&0
    \end{array}
\right)
\left(\begin{array}{c}U \\P \end{array}\right)=\left(\begin{array}{c}
F_1 \\ F_2
 \end{array}
 \right).
\end{equation}

The methodology of implementing this weak Galerkin method is the same as that for continuous
Galerkin finite element method except that computing  standard gradient $\nabla$ and divergence $\nabla\cdot$ are replaced by computing weak gradient $\nabla_w$ and weak divergence $\nabla_w\cdot$. For basis function $\Theta_l$, we will show that $\nabla_w\cdot\Theta_l$ and $\nabla_w\Theta_l$ can be calculated explicitly.

The procedures of implementing the method (\ref{wg-m})-(\ref{wg-c}) can be described as following steps. Here we let $d=2$ for simplicity.

\medskip

1. Find basis functions for $V_h$ and $W_h$. First we define two types of scalar piecewise constant basis functions $\phi_i$ associated with the interior of the triangle $T_i\in\T_h$  and $\psi_j$ associated with an edge $e_j\in \E_h$ respectively,
   \begin{equation}\label{phi}
\phi_i=\left\{
\begin{array}{l}
  1\quad\mbox{on} \;\; T_i^0,\\ [0.08in]0\quad\mbox{therwise},
\end{array}
\right.
\psi_j=\left\{
\begin{array}{l}
  1\quad\mbox{on} \;\; e_j.\\ [0.08in]0\quad\mbox{therwise},
\end{array}
\right.
\end{equation}
Please note that $\phi_i$ and $\psi_j$ are functions defined over the whole domain $\Omega$. Then we can define the vector basis functions for velocity as
\begin{equation}\label{b1}
\Phi_{i,1}=\left(\begin{array}{c}\phi_i \\0 \end{array}\right),\;\;\Phi_{i,2}=\left(\begin{array}{c}0 \\\phi_i \end{array}\right),\;\; i=1,\cdots,N_T,
\end{equation}
 and
\begin{equation}\label{b2}
\Psi_{i,1}=\left(\begin{array}{c}\psi_i \\0 \end{array}\right),\;\;\Psi_{i,2}=\left(\begin{array}{c}0 \\\psi_i \end{array}\right), \;\; i=1,\cdots,N_E.
\end{equation}
Let $n=N_T+N_E$. These $2n$ vector functions will form a basis for $V_h$,
\begin{eqnarray}
V_h&=&{\rm span} \{\Phi_{i,j}, i=1,\cdots, N_T,\;\Psi_{k,j}, k=1,\cdots,N_E, \;j=1,2\}\nonumber\\
&=&{\rm span} \{\Theta_1,\cdots,\Theta_{2n}\}.\label{vh}
\end{eqnarray}
Define $\overline{W}_h$ as
\[
\overline{W}_h={\rm span} \{\phi_1,\cdots,\phi_{N_T}\}.
\]
The pressure space is a subspace of $\overline{W}_h$,
\[
W_h=\{q\in \overline{W}_h, \int_\Omega qdx=0\}.
\]

\medskip

2. Compute weak gradient $\nabla_w$ and weak divergence $\nabla_w\cdot$ for the basis function $\Theta_l$ defined in (\ref{vh}). By the definition of $\Theta_l$,  to compute $\nabla_w\cdot\Theta_l$, we compute $\nabla_w\cdot\Phi_{i,j}$ and $\nabla_w\cdot\Psi_{i,j}$ instead. To find $\nabla_w\Theta_l$, we just need to compute $\nabla_w\phi_i$ and $\nabla_w\psi_i$.

\medskip

\begin{itemize}
\item Computing $\nabla_w\cdot\Phi_{i,j}$.\\
Using (\ref{dwd}), we have $\nabla_w\cdot\Phi_{i,j}|_T=0$ for all $T\in\T_h$.\\

\item Computing $\nabla_w\cdot\Psi_{i,j}$.\\
 Assume that $i^{th}$ edge $e_i$ is  on $\pT$ and $\Psi_{i,j}$ is defined in (\ref{b2}).  Then 
\begin{equation}\label{cwd}
\nabla_w\cdot\Psi_{i,j}|_T=\frac1{|T|}\int_{e_i}\Psi_{i,j,b}\cdot\bn ds,
\end{equation}
where $|T|$ is the area of $T$ and $\Psi_{i,j}=\{\Psi_{i,j,0},\Psi_{i,j,b}\}$.
Note that $\nabla_w\cdot\Psi_{i,j}$ is only nonzero  on two triangles that share $e_i$.\\

\item Computing $\nabla_w\phi_i$.\\
  Let $T$ be the $i^{th}$ triangle in $\T_h$ and $\phi_i$ be defined as in (\ref{phi}). Then $\nabla_w\phi_i$ is only nonzero on $T$ and can be calculated by
 \begin{eqnarray*}
\nabla_w\phi_i|_T&=&-C_T(\bx-\bx_T),
\end{eqnarray*}    
where $\bx_T$ is the centroid  of $T$ and $C_T= \frac{2|T|}{\|\bx-\bx_T\|_T^2}$.\\

\item Computing $\nabla_w\psi_i$.\\
 Assume that $i^{th}$ edge $e_i$ is  on $\pT$ and $\psi_{i}$ is defined in (\ref{phi}). Then
\begin{eqnarray*}
\nabla_w\psi_i|_T&=&\frac{C_T}{3}(\bx-\bx_T)+\frac{|e_i|}{|T|}\bn,
\end{eqnarray*}
Note that $\nabla_w\psi_i$ is only nonzero  on  two triangles that share $e_i$.
\end{itemize}

\medskip

3. Form the stiffness matrix (\ref{matrix}) with
\[
A=(a_{ij})=(a(\Theta_i,\Theta_j)),\quad B=(b_{ij})=(b(\Theta_i,\phi_j)).
\]
Note that
$$
a(\Theta_i,\Theta_j)=\sum_{T\in\T_h}(\nabla_w\Theta_i,\nabla_w\Theta_j)_T,\;\; b(\Theta_i,\phi_j)=\sum_{T\in\T_h}(\nabla_w\cdot\Theta_i,\phi_j)_T.
$$

\section{Error estimate}\label{section-error}
Denote by $\pi_h$ a $L^2$ projection from  $[L^2(T)]^{d\times d} $ to $[RT_0(T)]^d$.
We also define a projection $\Pi_h$ such that
$\Pi_h\bq\in [H({\rm div},\Omega)]^d$, and on each $T\in {\cal T}_h$,
one has $\Pi_h\bq \in [RT_0(T)]^d$ and the following equation satisfied:
$$
(\nabla\cdot\bq,\;\bv_0)_T=(\nabla\cdot\Pi_h\bq,\;\bv_0)_T, \qquad
\forall \bv_0\in [P_0(T)]^d.
$$
For any $\tau\in  [H({\rm div},\Omega)]^d$,  we have (see \cite{bf})
\begin{equation}\label{4.200}
\sum_{T\in {\cal T}_h}(-\nabla\cdot\tau, \;\bv_0)_T=\sum_{T\in {\cal T}_h}(\Pi_h\tau, \;\nabla_w\bv)_T,\quad\forall \bv=\{\bv_0,\bv_b\}\in V_h.
\end{equation}
The following two identities can be verified easily and also can be found in \cite{wy, wy-mixed}.
\begin{eqnarray}
\nabla_w \bQ_h \bu &=& \pi_h (\nabla \bu),\label{4.88}\\
\nabla_w\cdot \bQ_h \bu &=& Q_0 (\nabla \cdot\bu).\label{4.99}
\end{eqnarray}

We introduce two semi-norms $\3bar \bv\3bar$ and $\|\cdot\|_{1,h}$ as follows:
\begin{eqnarray}
\3bar \bv\3bar^2 &=& a(\bv,\bv), \label{norm2}\\
\|\bv\|_{1,h}^2&=&\sum_{T\in \T_h}\left(\|\nabla\bv_0\|_T^2+h_T^{-1}\|\bv_0-\bv_b\|_{\pT}^2\right).\label{norm3}
\end{eqnarray}

The following norm equivalences is proved in \cite{mwwy} that there exist two constants $C_1$ and $C_2$ independent of $h$ satisfying
\begin{equation}\label{norm-e}
C_1\|\bv\|_{1,h}\le \3bar\bv\3bar\le C_2 \|\bv\|_{1,h}.
\end{equation}

\begin{lemma}
The semi-nome $\3bar\cdot\3bar$ defined in (\ref{norm2}) is a norm in $V_h^0$.
\end{lemma}

\medskip

\begin{proof}
We only need to prove $v=0$ if $\3bar v\3bar=0$ for all $v\in V_h^0$. Let $v\in V_h^0$ and $\|v\|_{1,h}=0$. Then we have $\nabla v_0=0$ on each $T\in\T_h$ , $v_0=v_b$ on $e\in\E_h^0$ and $v_b=0$ for $e$ on $\partial\Omega$. $\nabla v_0=0$ on $T$ implies that $v_0$ is a constant on each $T$. $v_0=v_b$ on $e$ means that $v_0$ is continuous. With $v_0=v_b=0$ on $\partial\Omega$, we conclude $v=0$ and prove that $\|\cdot\|_{1,h}$ is a norm in $V_h$. Combining it with (\ref{norm-e}), we have proved that $\3bar\cdot\3bar$ is a norm in $V_h^0$.
\end{proof}

\medskip

Define two linear functionals on $V_h^0$  by
\begin{eqnarray*}
\ell_{\bu}(\bv)=\sum_{T\in\T_h}(\Pi_h\nabla \bu-\pi_h\nabla\bu, \nabla_w\bv)_T,\;\;\;
\ell_p(\bv)=\sum_{T\in\T_h}\l \bv_0-\bv_b,\;(p-Q_0p)\bn\r_\pT.
\end{eqnarray*}

\begin{lemma}\label{e-e}
Let $\be_h=\bQ_h\bu-\bu_h=\{\be_0,\;\be_b\}=\{\bQ_0\bu-\bu_0,\;\bQ_b\bu-\bu_b\}$ and $\varepsilon_h=Q_0p-p_h$.
Then, the following equations hold
true
\begin{eqnarray}
a(\be_h,\bv)-b(\bv,\varepsilon_h)&=&-\ell_{\bu}(\bv)-\ell_p(\bv),\;\forall\bv\in V_h^0,\label{ee-m}\\
b(\be_h,q)&=&0,\quad\qquad\qquad\forall q\in W_h.\label{ee-c}
\end{eqnarray}
\end{lemma}

\begin{proof}
Testing (\ref{moment}) by $\bv=\{\bv_0,\bv_b\}\in V_h^0$ gives
\begin{equation}\label{mmm1}
(-\nabla\cdot(\nabla\bu), \;\bv_0)+(\nabla p,\;\bv_0)=(\bbf, \bv_0).
\end{equation}
Equations (\ref{4.200}) and (\ref{4.88})  imply
\begin{eqnarray*}
\sum_{T\in\T_h}(-\nabla\cdot(\nabla\bu), \;\bv_0)_T&=&\sum_{T\in\T_h}((\Pi_h\nabla\bu, \nabla_w\bv)_T=\sum_{T\in\T_h}((\Pi_h\nabla \bu-\pi_h\nabla\bu, \nabla_w\bv)_T+(\nabla_w\bQ_h\bu, \nabla_w\bv)_T).
\end{eqnarray*}
It follows from (\ref{dwd}), the integration by parts and the fact $\sum_{K\in\T_h}\langle\bv_b,\; p\;\bn\rangle_{\partial K}=0$ that for $v\in V_h^0$
\begin{eqnarray*}
(\nabla p,\;\bv_0)&&=\sum_{T\in\T_h}(-(p,\; \nabla\cdot\bv_0)_T +\l p\bn,\; \bv_0\r_T)\\
&&=\sum_{T\in\T_h}(-(Q_0p,\; \nabla\cdot\bv_0)_T +\l p\bn,\; \bv_0-\bv_b\r_T)\\
&&=\sum_{T\in\T_h}((\bv_0,\; \nabla Q_0p)_T -\l Q_0p\bn,\; \bv_0\r_T+\l p\bn,\; \bv_0-\bv_b\r_T)\\
&&=-b(\bv, Q_0p)+\sum_{T\in\T_h}\l \bv_0-\bv_b,\;(p-Q_0p)\bn\r_\pT.
\end{eqnarray*}
Combining two equations above with (\ref{mmm1}), we have
\begin{eqnarray}
a(\bQ_h\bu,\bv)-b(\bv,Q_0p)=(\bbf, \bv_0)-\ell_{\bu}(\bv)-\ell_p(\bv).\label{t-m}
\end{eqnarray}
Using (\ref{4.99}) and (\ref{cont}), we have that for any $q\in W_h$
\begin{eqnarray}
b(\bQ_h\bu,q)&=&\sum_{T\in\T_h}(\nabla_w\cdot\bQ_h\bu, q)_T=\sum_{T\in\T_h}(Q_0(\nabla\cdot\bu), q)_T\nonumber\\
&=&\sum_{T\in\T_h}(\nabla\cdot\bu, q)_T=b(\bu,q)=0,\label{t-c}
\end{eqnarray}
The differences of (\ref{wg-m})-(\ref{wg-c}) and (\ref{t-m})-(\ref{t-c}) yield (\ref{ee-m})-(\ref{ee-c}). The proof is completed.
\end{proof}

\medskip

\begin{lemma}
For any $\rho\in W_h$, then there exists a constant $C$ independent of $h$ such that
\begin{equation}\label{inf-sup}
\sup_{\bv\in V_h^0}\frac{(\nabla_w\cdot\bv,\;\rho)}{\3bar\bv\3bar}\ge C\|\rho\|.
\end{equation}
\end{lemma}

\begin{proof}
For a given $\rho\in W_h\subset L_0^2(\Omega)$, it is well known that there exists  $\tilde\bv\in [H_0^1(\Omega)]^d$ such that
\begin{equation}\label{c-inf-sup}
\frac{(\nabla\cdot\tilde\bv,\rho)}{\|\nabla\tilde\bv\|}\ge C\|\rho\|.
\end{equation}
Let $\bv=\bQ_h\tilde{\bv}$.
(\ref{4.88}) implies
\begin{equation}\label{m2}
\3bar\bv\3bar=\|\nabla_w\bv\|=\|\nabla_w(\bQ_h\tilde{\bv})|=\|\pi_h\nabla \tilde{\bv}\|\le \|\nabla\tilde\bv\|.
\end{equation}
It follows from (\ref{4.99}) and the definition of $\pi_h$
\begin{eqnarray*}
(\nabla_w\cdot\bv,\;\rho)&=&(\nabla_w\cdot \bQ_h\tilde\bv,\;\rho)=(Q_0(\nabla\cdot\tilde\bv),\;\rho)=(\nabla\cdot\tilde\bv,\;\rho).
\end{eqnarray*}
Using the equation above, (\ref{c-inf-sup}) and (\ref{m2}), we have
\begin{eqnarray*}
\frac{(\nabla_w\bv,\rho)} {\3bar\bv\3bar} &\ge & \frac{(\nabla\cdot\tilde\bv,\rho)}{\|\nabla\tilde\bv\|}\ge C\|\rho\|.
\end{eqnarray*}
We proved the lemma.
\end{proof}

For any function $\varphi\in H^1(T)$, the following trace
inequality holds true 
\begin{equation}\label{trace}
\|\varphi\|_{e}^2 \leq C \left( h_T^{-1} \|\varphi\|_T^2 + h_T
\|\nabla \varphi\|_{T}^2\right).
\end{equation}

\begin{theorem}\label{h1-bd}
Let $(\bu, p)\in  [H_0^1(\Omega)\cap H^{2}(\Omega)]^d\times L_0^2(\Omega)\cap H^1(\Omega)$  and $(\bu_h,p_h)\in V_h\times W_h$ be the solutions of (\ref{moment})-(\ref{bc}) and (\ref{wg-m})-(\ref{wg-c}) respectively. Then
\begin{eqnarray}
\|\nabla_w (\bQ_h\bu- \bu_h)\|+\|Q_0p-p_h\|&\le& Ch(\|\bu\|_{2}+\|p\|_1),\label{error1}\\
\|\bQ_0\bu-\bu_0\|&\le& Ch^2(\|\bu\|_{2}+\|p\|_1).\label{error2}
\end{eqnarray}
\end{theorem}

\begin{proof}
Letting $\bv=\be_h$ and $q=\varepsilon_h$ in (\ref{ee-m})-(\ref{ee-c}) and adding the two equations yield
\begin{equation}\label{eee1}
\3bar\be_h\3bar^2=a(\be_h,\be_h)=|\ell_{\bu}(\be_h)+\ell_{p}(\be_h)|.
\end{equation}
It follows from the definitions of $\Pi_h$ and $\pi_h$ that
\begin{eqnarray}
|\ell_{\bu}(\be_h)|&=&|\sum_{T\in\T_h}(\Pi_h\nabla \bu-\pi_h\nabla\bu, \nabla_w\be_h)_T|\nonumber\\
&\le& \sum_{T\in\T_h}\|\Pi_h\nabla \bu-\pi_h\nabla\bu\|_T\|\nabla_w\be_h\|_T\|\le Ch\|\bu\|_2\3bar\be_h\3bar.\label{eee2}
\end{eqnarray}
Using the trace inequality (\ref{trace}), the definition of $Q_0$ and norm equivalence (\ref{norm-e}), we have
\begin{eqnarray}
|\ell_{p}(\be_h)|&=&|\sum_{T\in\T_h}\l \be_0-\be_b,\;(p-Q_0p)\bn\r_\pT|\nonumber\\
&\le&\sum_{T\in\T_h}\|p-Q_0p\|_{\pT}\|\be_0-\be_b\|\nonumber\\
&\le&(\sum_{T\in\T_h}h\|p-Q_0p\|_{\pT}^2)^{1/2}(\sum_{T\in\T_h}h^{-1}\|\be_0-\be_b\|_\pT^2)^{1/2}\nonumber\\
&\le& Ch\|p\|_1\3bar\be_h\3bar.\label{eee3}
\end{eqnarray}
Combining the estimates (\ref{eee2})-(\ref{eee3}) and (\ref{eee1}) gives
\begin{equation}\label{eee4}
\3bar\be_h\3bar\le Ch(\|\bu\|_2+\|p\|_1).
\end{equation}
Using (\ref{ee-m}) and (\ref{eee2})-(\ref{eee4}), we have
\begin{equation}\label{eee5}
|b(\bv,\varepsilon_h)|=|a(\be_h,\bv)+\ell_{\bu}(\bv)+\ell_p(\bv)|\le Ch(\|\bu\|_2+\|p\|_1)\3bar\bv\3bar.
\end{equation}

It follows from (\ref{inf-sup}) and (\ref{eee5})
\[
\|Q_0p-p_h\|\le Ch(\|\bu\|_{2}+\|p\|_1).
\]
The estimate (\ref{error2}) can be derived by using the standard duality argument. The main goal  of this paper is about introducing a new method and how to  implement it. We skip the proof of (\ref{error2}).
\end{proof}

\section{Discrete Divergence Free Basis}\label{section-divfree}

The finite element formulations (\ref{wg-m})-(\ref{wg-c}) lead to a large saddle point system (\ref{matrix})
for which most existing numerical solvers are less effective and robust than for definite
systems. Such a saddle-point system can be reduced to a definite problem
for the velocity unknown defined in a divergence-free subspace $D_h$ of $V_h^0$,
\begin{equation}\label{D_h}
D_h=\{\bv\in V_h;\;b(\bv,q) =0,\quad\forall q\in W_h\}.
\end{equation}

By taking the test function from $D_h$,
the discrete formulation (\ref{wg-m})-(\ref{wg-c}) is equivalent to the following  divergence-free finite element scheme:

\medskip

\begin{algorithm}\label{algorithm2}
A discrete divergence free approximation for (\ref{moment})-(\ref{bc}) with homogeneous boundary condition  is to find
$\bu_h=\{\bu_0,\bu_b\}\in D_h$
\begin{eqnarray}
a(\bu_h,\ \bv)&=&(f,\;\bv_0),\quad \forall \bv=\{\bv_0,\bv_b\}\in D_h.\label{df-wg}
\end{eqnarray}
\end{algorithm}

The system (\ref{df-wg}) is symmetric and positive definite with much fewer unknowns. To implement Algorithm \ref{algorithm2}, we need to construct basis functions for the divergence free subspace $D_h$. There are three types of functions in $V_h$ that are  divergence free.

\medskip

{\bf Type 1}. All $\Phi_{i,j}$ defined in (\ref{b1}) are divergence free. This can be verified easily. Since all the functions $\Phi_{i,j}$ defined in (\ref{b1}) take zero value on $\pT$ for all $T\in\T_h$, it follows from (\ref{cwd}) that $\nabla_w\cdot\Phi_{i,j}|_T=\frac1{|T|}\int_\pT\Phi_{i,j,b}\cdot\bn ds=0$.

{\bf Type 2}. For any $e_i\in\E_h^0$, let $\bn_{e_i}$ and $\bt_{e_i}$ be a normal vector and a tangential vector to $e_i$ respectively. Define  $\Upsilon_i=C_1\Psi_{i,1}+C_2\Psi_{i,2}$ such that $\Upsilon_i|_{e_i}=\bt_{e_i}$. Thus $\Upsilon_i$ is only nonzero on $e_i$. It is easy to see that  $\nabla_w\cdot\Upsilon_i|_T=\frac1{|T|}\int_\pT\Upsilon_{i,b}\cdot\bn ds=0$.

{\bf Type 3}. For a given interior vertex $P_i\in {\cal V}_h$, there are $r$ elements having $P_i$ as a vertex which form a hull ${\cal H}_{P_i}$ as shown in Figure \ref{fig1}. Then there are $r$ interior edges $e_j$ ($j=1,\cdots, r$) associated with ${\cal H}_{P_i}$. Let $\bn_{{e_j}}$ be a normal vector on $e_j$ such that normal vectors $\bn_{e_j}$ $j=1,\cdots,r$ are counterclockwise around vertex $P_i$ as shown in Figure \ref{fig1}.  For each $e_j$, let $\Psi_{j,1}$ and $\Psi_{j,2}$ be the two basis functions of $V_h$ which is only nonzero on $e_j$. Define $\Theta_j=C_1\Psi_{j,1}+C_2\Psi_{j,2}\in V_h$ such that $\Theta_j|_{e_j}=\bn_{e_j}$. Define $\Lambda_i=\sum_{j=1}^r \frac{1}{|e_j|}\Theta_j$. It can be shown that $\nabla_w\cdot\Lambda_i=0$ (see \cite{mwy-divfree} for the details)

\begin{figure}
\begin{center}
\includegraphics[width=4cm]{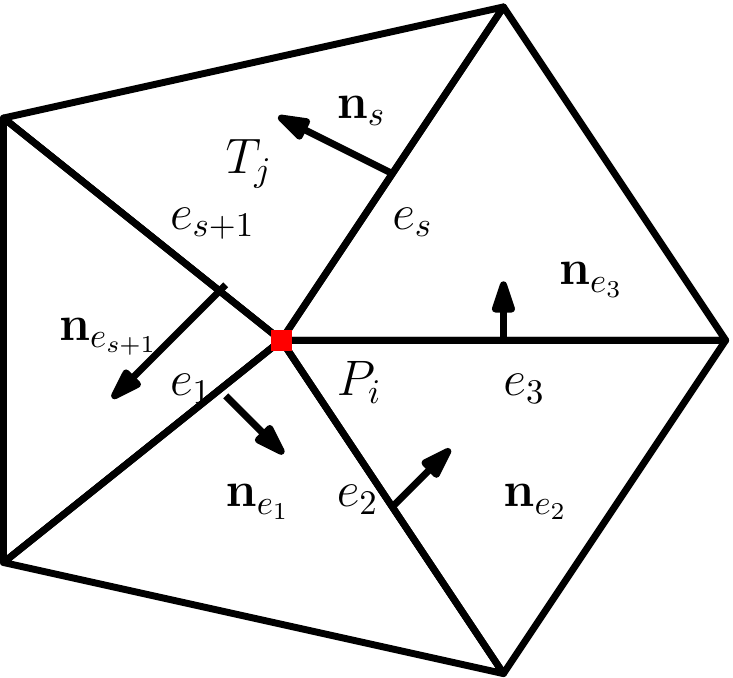}
\caption{Hull ${\cal H}_{P_i}$ .}  \label{fig1}
\end{center}
\end{figure}

\medskip

\begin{lemma}
These three types of divergence free functions form a  basis for $D_h$, i.e.
\begin{equation}\label{b5}
D_h={\rm span}\{\Phi_{i,j},i=1,\cdots,N_T,j=1,2;\Upsilon_i, i=1,\cdots, N_E; \Lambda_i, i=1,\cdots, N_V\}.
\end{equation}
\end{lemma}

\begin{proof}
It is easy to check that all the weakly divergence free functions in (\ref{b5}) are linearly independent. It is left to check that the number of the basis functions in (\ref{b5}) is equal to the dimension of $D_h$.   Obviously, the number of the functions in (\ref{b5}) is $2N_T+N_E+N_V$. On the other hand we have
$$\rm{dim} (D_h)=\rm{dim} (V_h)-\rm{dim}(W_h)=2N_T+2N_E-N_T+1.$$
For $\T_h$, it is well known as Euler formula  that
\begin{equation}\label{key}
N_E+1=N_V+N_K.
\end{equation}
Using (\ref{key}), we have
$$\rm{dim} (D_h)=\rm{dim} (V_h)-\rm{dim}(W_h)=2N_T+2N_E-N_T+1=2N_T+N_E+N_V.$$
We have proved the lemma.
\end{proof}

\section{Numerical Examples}\label{section-ne}

In this section, six numerical examples are tested  for
the two dimensional Stokes equations (\ref{moment})-(\ref{bc}).  The numerical experiments indicate that the weak Galerkin methods are robust, accurate and easy to implement.

\subsection{Example 1}\label{Num_ex1}
We first consider the Stokes equations with homogeneous boundary condition defined on a square $(0, 1)\times (0, 1)$.
The exact solutions are given by
$$
\bu = \begin{pmatrix}2\pi\sin(\pi x)\sin(\pi x)\cos(\pi y)\sin(\pi y) \\
-2\pi\sin(\pi x)\sin(\pi y)\cos(\pi x)\sin(\pi y)
\end{pmatrix},
$$
and
\[
p=\cos(\pi x)\cos(\pi y).
\]
The uniform triangular mesh is used for testing. Denote mesh size by $h.$ The numerical results of Algorithm 1 are presented in Table \ref{ex1_1}. These results show the $O(h)$ error of the velocity in the $H^1-$norm and pressure in the $L^2-$norm as predicted by Theorem \ref{h1-bd}. Convergence rate of $O(h^2)$ for velocity in the $L^2-$norm is observed.

Furthermore, the divergence free weak Galerkin Algorithm 2 is tested for this example. The weakly divergence-free subspace $D_h$ is constructed as described in previous section. By using the basis functions in $D_h$, the saddle-point system (\ref{wg-m})-(\ref{wg-c}) is reduced to a definite system (\ref{df-wg}) only depending on velocity unknowns. The numerical performance of velocity measured in $H^1-$norm and $L^2-$norm is shown in Table \ref{ex1_2}.

\begin{table}
  \caption{Example \ref{Num_ex1}. Numerical results of Algorithm 1.}
  \label{ex1_1}
  \center
   \begin{tabular}{||c||ccc||}
    \hline\hline
    $h$ & $\3bar \bu_h-\bQ_{h}\bu\3bar$ & $\|\bu_0-\bQ_{0}\bu\|$ & $||Q_0p-p_h||$ \\
    \hline\hline
  1/4   &4.0478     &3.7181e-1  &1.7906    \\ \hline
  1/8   &1.8723     &9.8624e-2  &8.7513e-1 \\ \hline
  1/16  &9.1907e-1  &2.5276e-2  &4.1211e-1 \\ \hline
  1/32  &4.5785e-1  &6.3793e-3  &2.0019e-1 \\ \hline
  1/64  &2.2874e-1  &1.5992e-3  &9.9207e-2 \\ \hline
  1/128 &1.1435e-1  &4.0009e-4  &4.9486e-2 \\ \hline\hline
   $O(h^r),r=$ &  1.0238  &   1.9750 &  1.0386   \\ \hline\hline
   \end{tabular}
\end{table}

\begin{table}
  \caption{Example \ref{Num_ex1}.  Numerical results of Algorithm 2.}
  \label{ex1_2}
  \center
   \begin{tabular}{||c||cc||}
    \hline\hline
    $h$ &$\3bar\bu_h-\bQ_{h}\bu\3bar$ & $\|\bu_0-\bQ_{0}\bu\|$ \\
    \hline\hline
  1/4   &6.3120    &2.6300e-1 \\ \hline
  1/8   &3.3499    &6.9789e-2 \\ \hline
  1/16  &1.7174    &1.7890e-2 \\ \hline
  1/32  &8.6696e-1 &4.5154e-3 \\ \hline
  1/64  &4.3468e-1 &1.1320e-3 \\ \hline
  1/128 &2.1750e-1 &2.8320e-4 \\ \hline\hline
  $O(h^r),r=$ &9.7484e-01 &1.9748 \\ \hline\hline
   \end{tabular}
\end{table}

\subsection{Example 2}\label{Num_ex2}
The purpose of this example is to test the robustness and accuracy of this WG method for handling  non-homogeneous boundary condition and complicated geometry.

Consider the  Stokes equations with non-homogeneous boundary condition that have the exact solutions
$$
\bu = \begin{pmatrix}x+x^2-2xy+x^3-3xy^2+x^2y \\
-y-2xy+y^2-3x^2y+y^3-xy^2
\end{pmatrix},
$$
and
\[
p=xy+x+y+x^3y^2-4/3.
\]

In this example, domain $\Omega$ is derived from a square $(0,1)\times(0,1)$ by taking out three circles  centered at $(0.5,0.5),\ (0.2,0.8),\ (0.8,0.8) $ with radius $0.1$.
We start the weak Galerkin simulation on the initial mesh as shown in Figure \ref{IniMesh_ex2}. Then each refinement is obtained by dividing each triangle into four congruent triangles.

Table \ref{tab:inhomogeneous0} displays the errors and convergence rate of the numerical solution of Algorithm 1. Optimal order convergence rates for velocity and pressure are observed in corresponding norms.

\begin{table}
  \caption{Example \ref{Num_ex2}.  Numerical results of Algorithm 1.}
  \label{tab:inhomogeneous0}
  \center
  \begin{tabular}{||c||ccc||}
    \hline\hline
    $h$ & $\3bar \bu_h-\bQ_{h}\bu\3bar$ & $\|\bu_0-\bQ_{0}\bu\|$ & $||Q_0p-p_h||$ \\
    \hline\hline
  Level 1  &1.5123e-1  &6.5055e-3  &2.2512e-1\\ \hline
  Level 2  &7.6271e-2  &1.7736e-3  &9.6785e-2\\ \hline
  Level 3  &3.8276e-2  &4.6167e-4  &4.0673e-2\\ \hline
  Level 4  &1.9168e-2  &1.1707e-4  &2.3700e-2\\ \hline
  Level 5  &9.5900e-3  &2.9394e-5  &1.9385e-2\\ \hline\hline
   $O(h^r),r=$ & 9.9506e-1  &  1.9501    & 9.1053e-1\\ \hline\hline
   \end{tabular}
\end{table}

\begin{figure}[!tb]
\centering
\begin{tabular}{c}
  \resizebox{2.45in}{2.1in}{\includegraphics{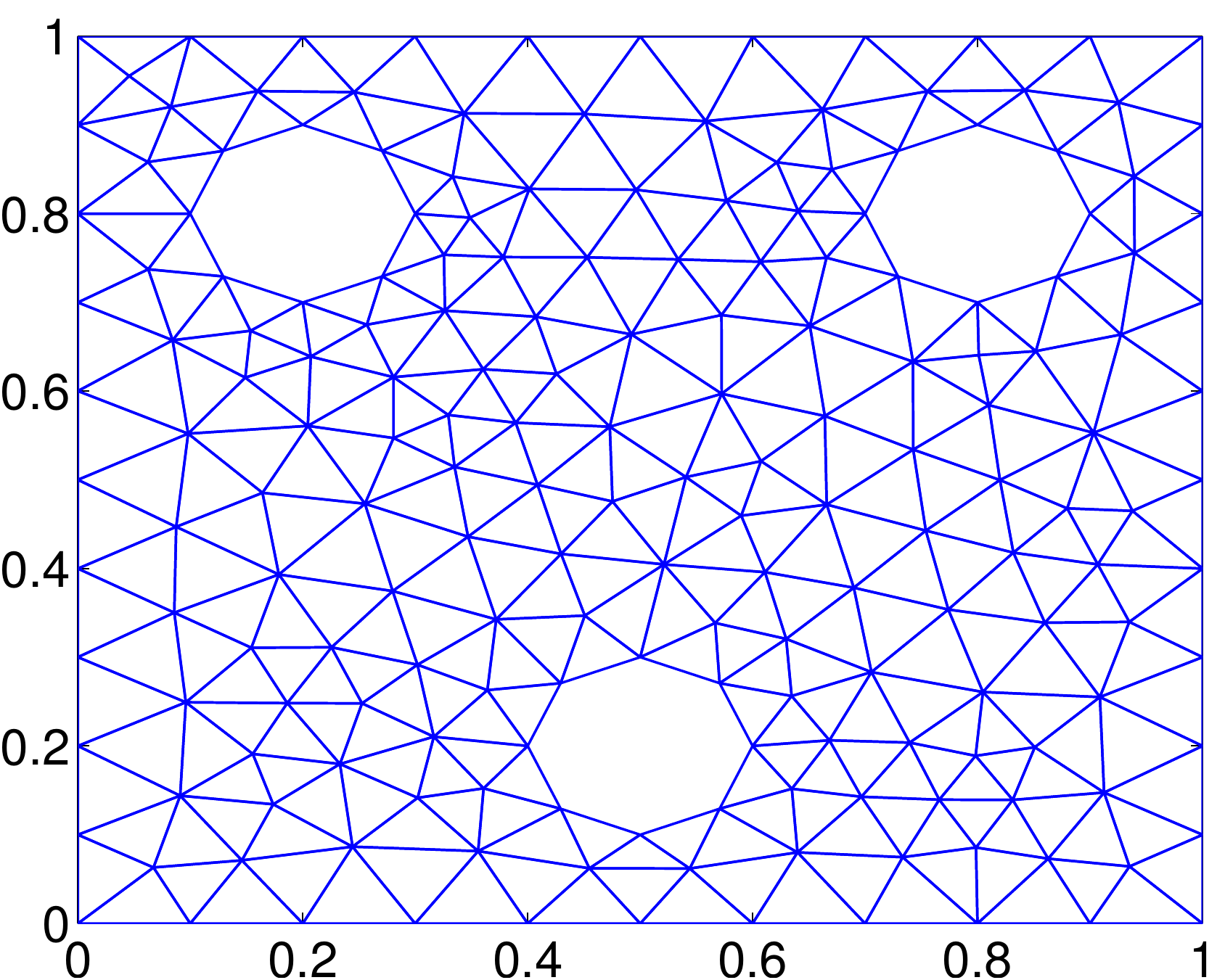}}
\end{tabular}
\caption{Example \ref{Num_ex2}: Initial mesh.}
\label{IniMesh_ex2}
\end{figure}

\subsection{Example 3}\label{Num_ex3}
Let $(\bu,p)$ as follows
\begin{eqnarray*}
\bu = \begin{pmatrix}1-e^{\lambda x}\cos(2\pi y) \\
\frac{\lambda}{2\pi}e^{\lambda x}\sin(2\pi y)
\end{pmatrix},
p=\frac{1}{2}e^{2\lambda x}+C,
\end{eqnarray*}
be the exact solution of the Stokes equations,
\begin{eqnarray}
-\frac{1}{\mathrm {Re}}\Delta\bu+\nabla p=\bbf,\quad \nabla\cdot\bu=0\mbox{  in }\Omega,
\end{eqnarray}
where $\lambda=\mathrm {Re}/2-\sqrt{\mathrm {Re}^2/4+4\pi^2}$ and $\mathrm {Re}$ is the Reynolds number.
Let $\Omega=(-1/2,3/2)\times(0,2).$  The Dirichlet boundary condition for velocity is considered in this example.

In Table \ref{ex3}, we demonstrate the error profiles and convergence rates of the numerical solution of Algorithm 1 with $\mathrm {Re}=1,\ 10,\ 100,\ 1000.$ The streamlines of velocity and color contour of pressure for $\mathrm {Re}=1,\ 10,\ 100,\ 1000$ are plotted in Figure \ref{fig:ex3}.
\begin{table}
  \caption{Example \ref{Num_ex3}:  Numerical results of Algorithm 1.}\label{ex3}
  \center
   \begin{tabular}{|c|cc|cc|cc|}
    \hline\hline
    $h$ & $\3bar\bu_h-\bQ_{h}\bu\3bar$ & order& $\|\bu_0-\bQ_{0}\bu\|$ & order &$||Q_0p-p_h||$ & order\\
    \hline
    \multicolumn{7}{|c|}{$\mathrm {Re}=1$}\\ \hline
   1/8   &4.2375e+1 &          &4.2372    &       &2.9223e+1&      \\
   1/16  &2.4722e+1 &7.7742e-1 &1.3963    &1.6015 &1.2713e+1&1.2008\\
   1/32  &1.3100e+1 &9.1623e-1 &3.9686e-1 &1.8149 &5.2018   &1.2892\\
   1/64  &6.6667e   &9.7452e-1 &1.0374e-1 &1.9357 &2.3142   &1.1685\\
   1/128 &3.3504e   &9.9264e-1 &2.6294e-2 &1.9802 &1.1550   &1.0026\\
    \hline
    \multicolumn{7}{|c|}{$\mathrm {Re}=10$}\\ \hline
   1/8   &6.0606    &          &2.0457    &       &7.6173e-1 &      \\
   1/16  &3.2851    &8.8352e-1 &5.9724e-1 &1.7762 &2.9472e-1 &1.3699\\
   1/32  &1.6896    &9.5926e-1 &1.5926e-1 &1.9069 &1.1379e-1 &1.3730\\
   1/64  &8.5296e-1 &9.8613e-1 &4.0832e-2 &1.9636 &4.5851e-2 &1.3113\\
   1/128 &4.2787e-1 &9.9531e-1 &1.0306e-2 &1.9862 &1.9770e-2 &1.2136\\
   \hline
    \multicolumn{7}{|c|}{$\mathrm {Re}=100$}\\ \hline
   1/8   &5.5209e-1 &          &6.7127e-1 &       &1.5818e-2 &      \\
   1/16  &2.7981e-1 &9.8046e-1 &1.7946e-1 &1.9032 &6.7914e-3 &1.2198\\
   1/32  &1.4063e-1 &9.9254e-1 &4.5955e-2 &1.9654 &2.9102e-3 &1.2226\\
   1/64  &7.0434e-2 &9.9756e-1 &1.1575e-2 &1.9892 &1.3386e-3 &1.1204\\
   1/128 &3.5235e-2 &9.9926e-1 &2.9001e-3 &1.9968 &6.4795e-4 &1.0468\\
   \hline
    \multicolumn{7}{|c|}{$\mathrm {Re}=1000$}\\ \hline
   1/8   &2.0636e-1 &          &8.1461e-1 &       &1.8694e-3 &      \\
   1/16  &1.0436e-1 &9.8359e-1 &2.1625e-1 &1.9134 &7.6097e-4 &1.2967\\
   1/32  &5.2395e-2 &9.9407e-1 &5.5149e-2 &1.9713 &3.2176e-4 &1.2419\\
   1/64  &2.6230e-2 &9.9821e-1 &1.3868e-2 &1.9916 &1.4850e-4 &1.1155\\
   1/128 &1.3120e-2 &9.9945e-1 &3.4726e-3 &1.9977 &7.2192e-5 &1.0406\\ \hline\hline
   \end{tabular}
\end{table}

\begin{figure}[!tb]
\centering
\begin{tabular}{cc}
  \resizebox{2.45in}{2.1in}{\includegraphics{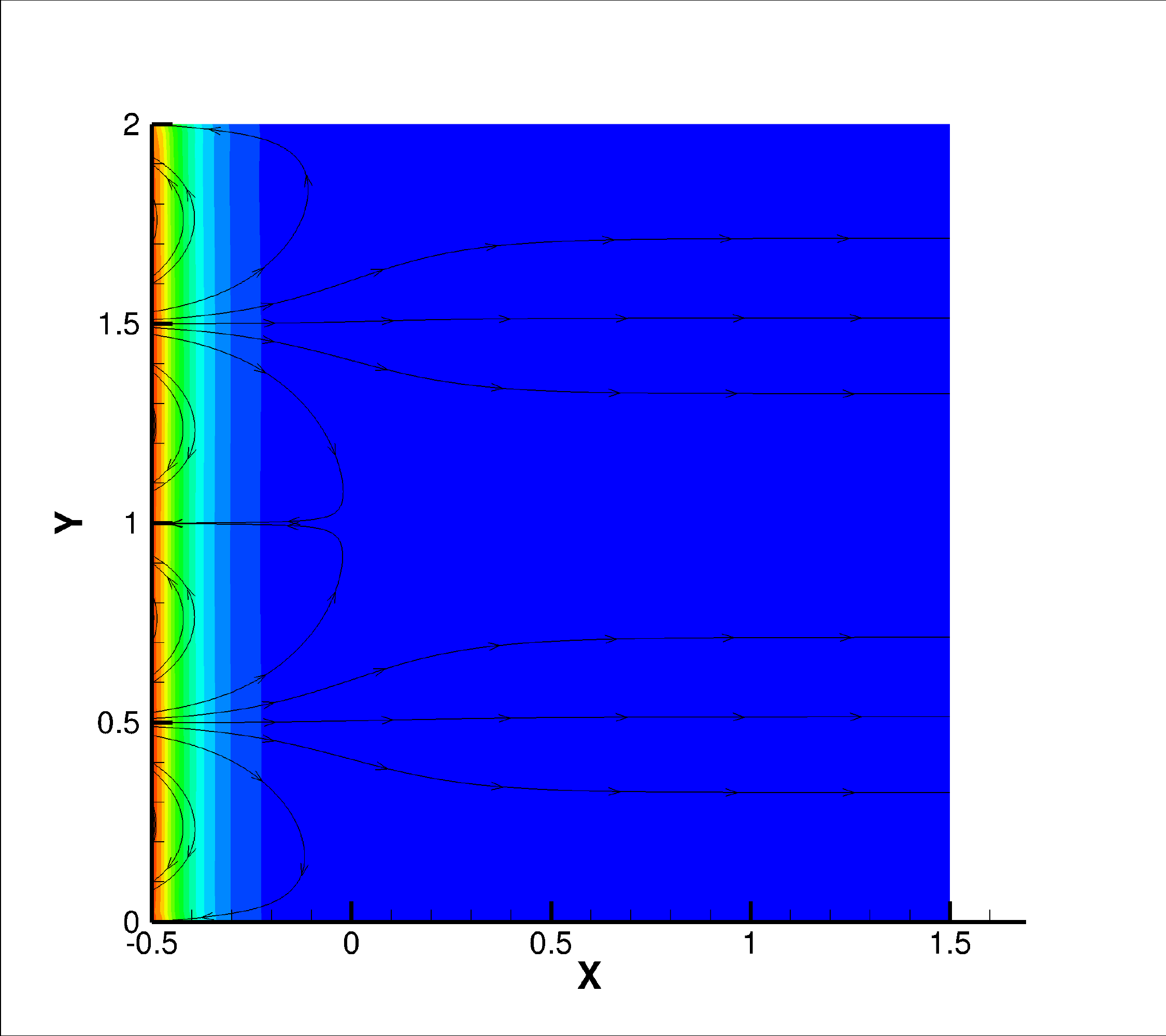}} \quad
  \resizebox{2.45in}{2.1in}{\includegraphics{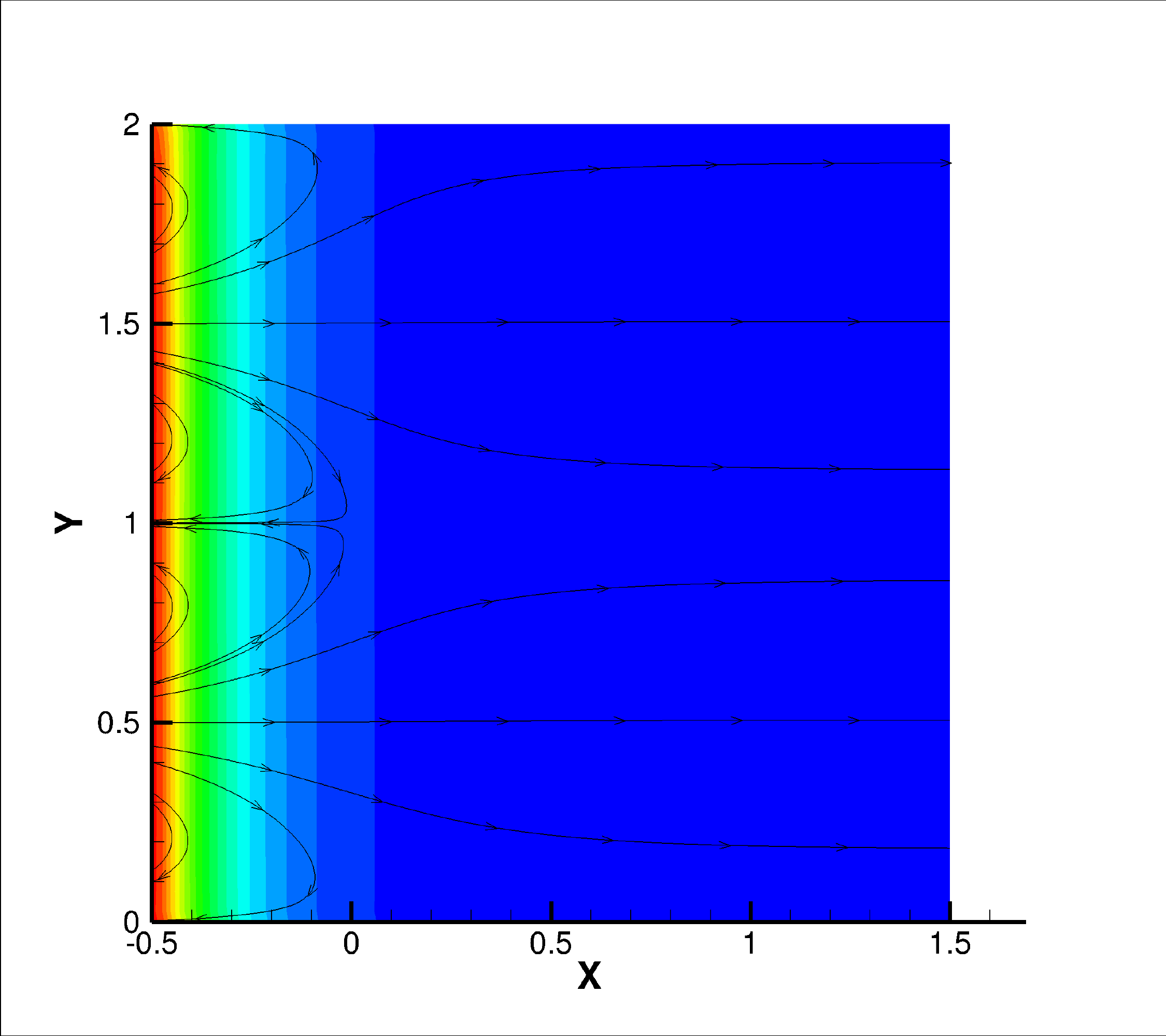}} \\
  \resizebox{2.45in}{2.1in}{\includegraphics{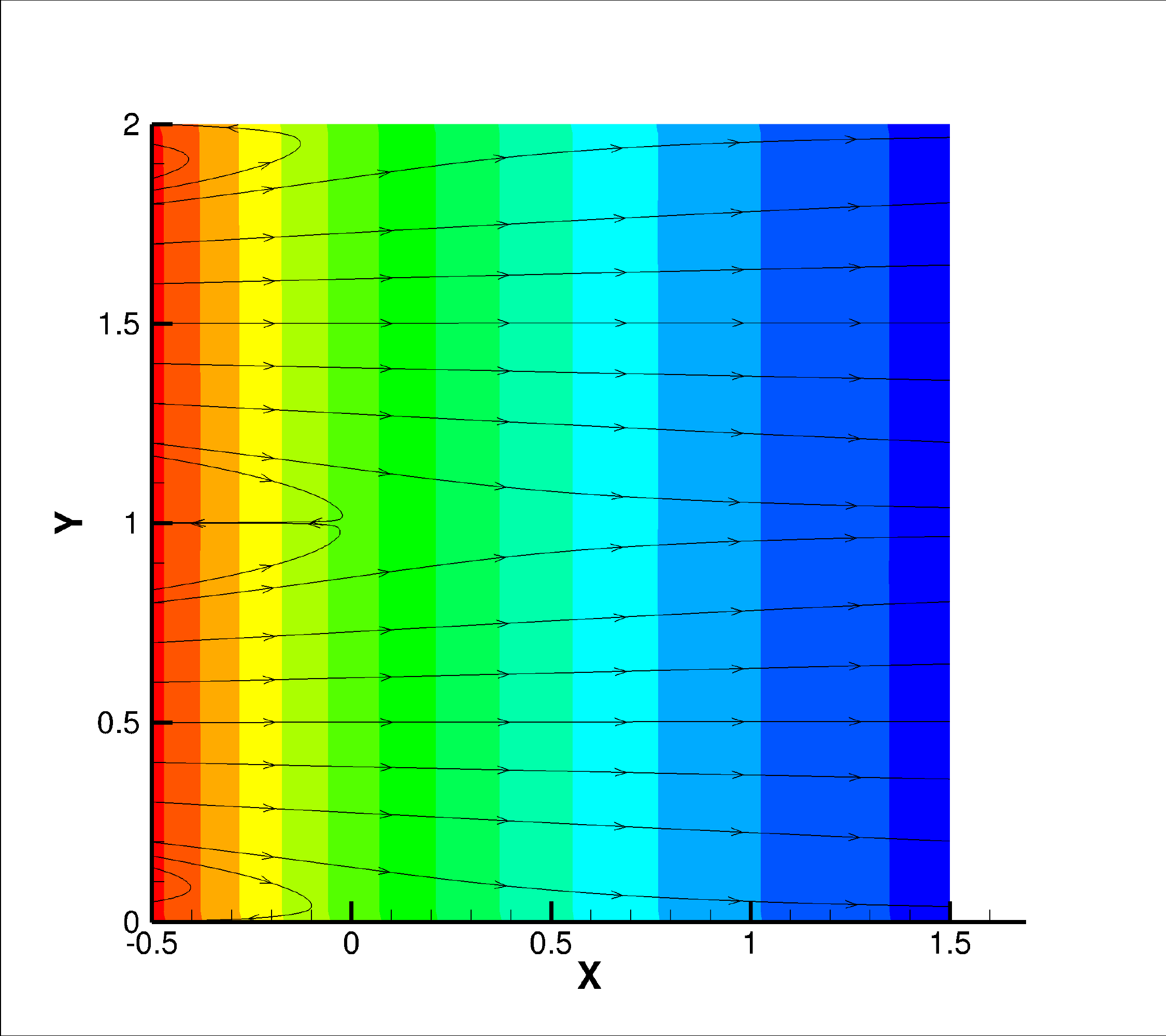}}\quad
  \resizebox{2.45in}{2.1in}{\includegraphics{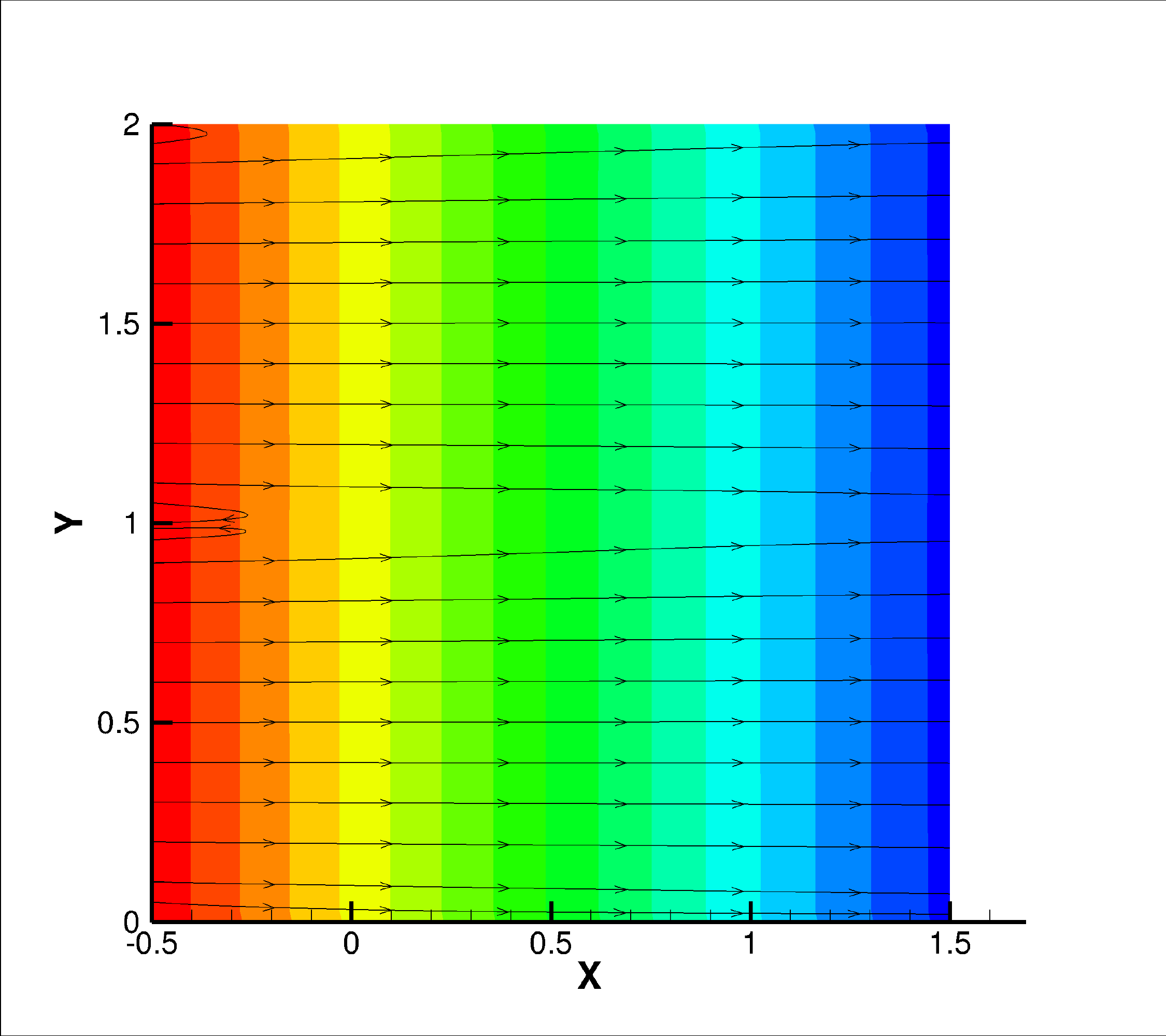}}
\end{tabular}
\caption{Example \ref{Num_ex3} Streamlines of velocity and color contour of pressure for $\mathrm {Re}=1,\ 10,\ 100,\ 1000$ (top left, top right, bottom left, bottom right.)}\label{fig:ex3}
\end{figure}

\subsection{Example 4}\label{Num_ex4}
Two dimensional channel flow around a circular obstacle is simulated in this problem.
We consider the Stokes equations with non-homogeneous boundary condition:
\begin{eqnarray*}
\bu|_{\partial\Omega}={\bf g}=\begin{cases}
(1,0)^t,&\mbox{ if } x=0;\\
(1,0)^t,&\mbox{ if } x=1;\\
(0,0)^t,&\mbox{ else. }
\end{cases}
\end{eqnarray*}
Let $\Omega=(0,1)\times(0,1)$ with one circular hole centered at $(0.5,0.5),$ with  radius 0.1.

We start with initial mesh and then refined the mesh uniformly. Level 1 mesh and level 2 mesh are shown in Figure \ref{IniMesh_Ex4}. The velocity fields of Algorithm 1  on level 1 and  level 2 meshes are shown in Figures \ref{Num_Ex4_1}. The streamlines of velocity and color contour of pressure are plotted in Figure \ref{Num_Ex4_2}.

\begin{figure}[!tb]
\centering
\begin{tabular}{cc}
  \resizebox{2.45in}{2.1in}{\includegraphics{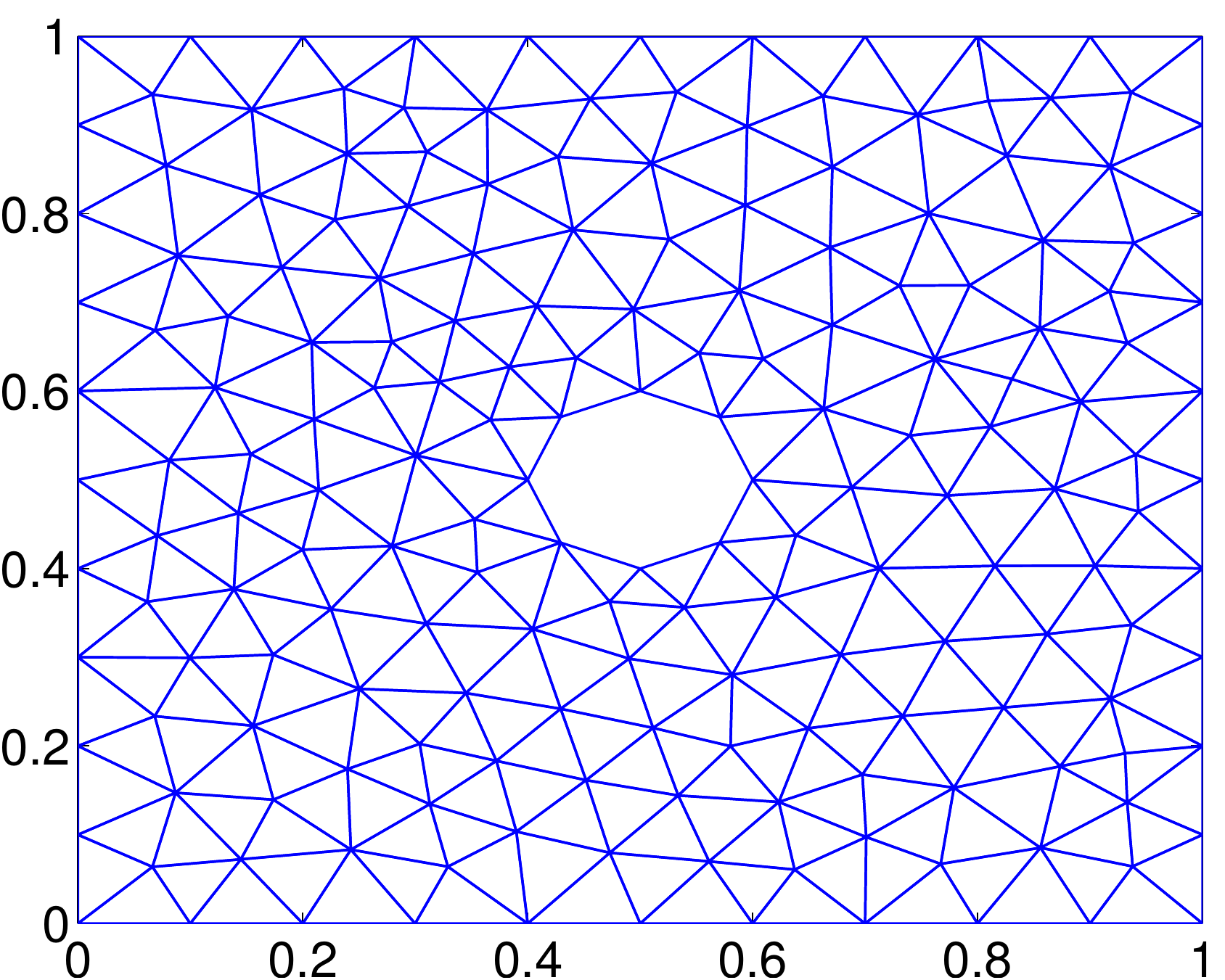}} \quad
  \resizebox{2.45in}{2.1in}{\includegraphics{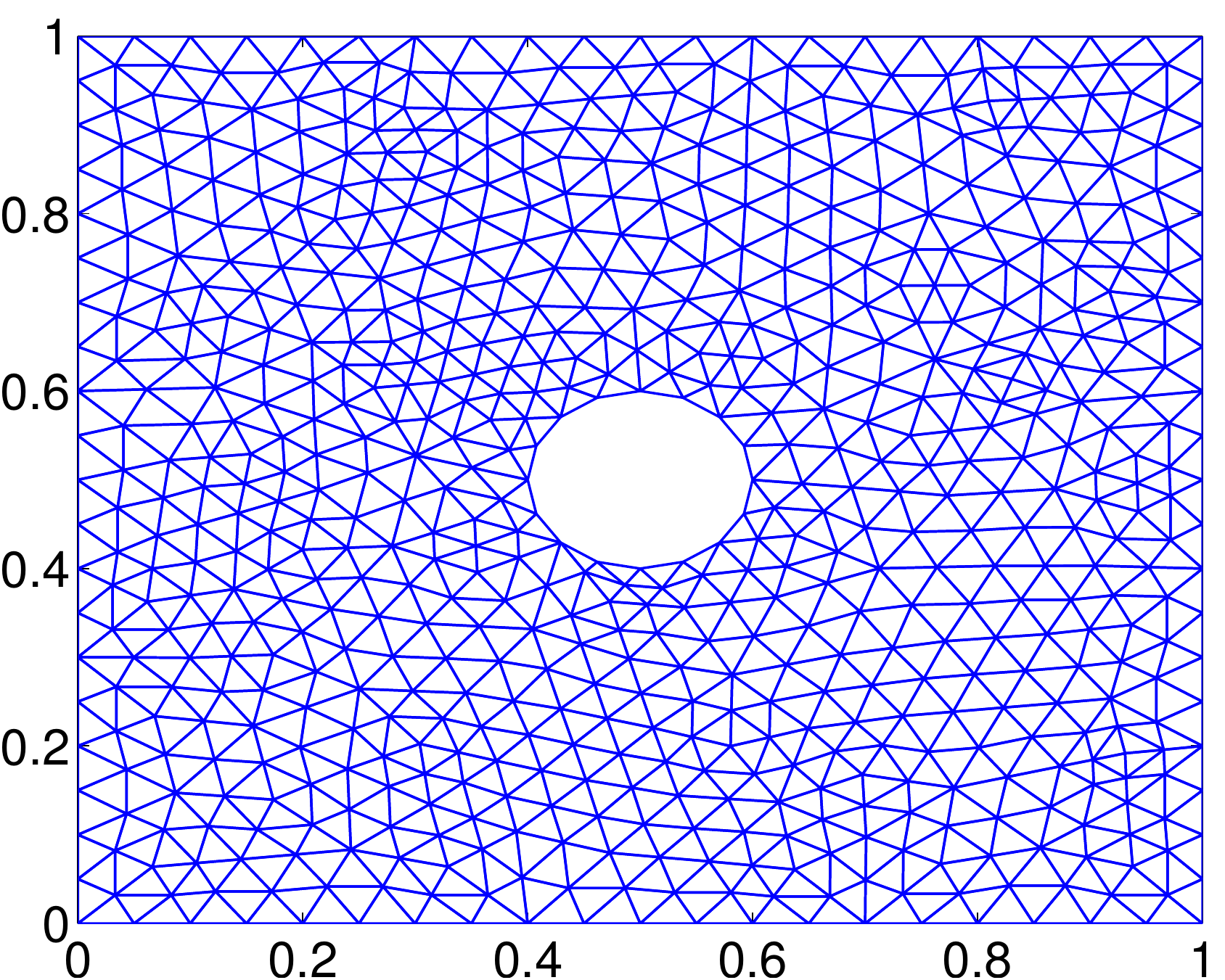}}
\end{tabular}
\caption{Example \ref{Num_ex4} Meshes of Level 1 (Left) and Level 2 (Right)}
\label{IniMesh_Ex4}
\end{figure}

\begin{figure}[!tb]
\centering
\begin{tabular}{cc}
  \resizebox{2.45in}{2.1in}{\includegraphics{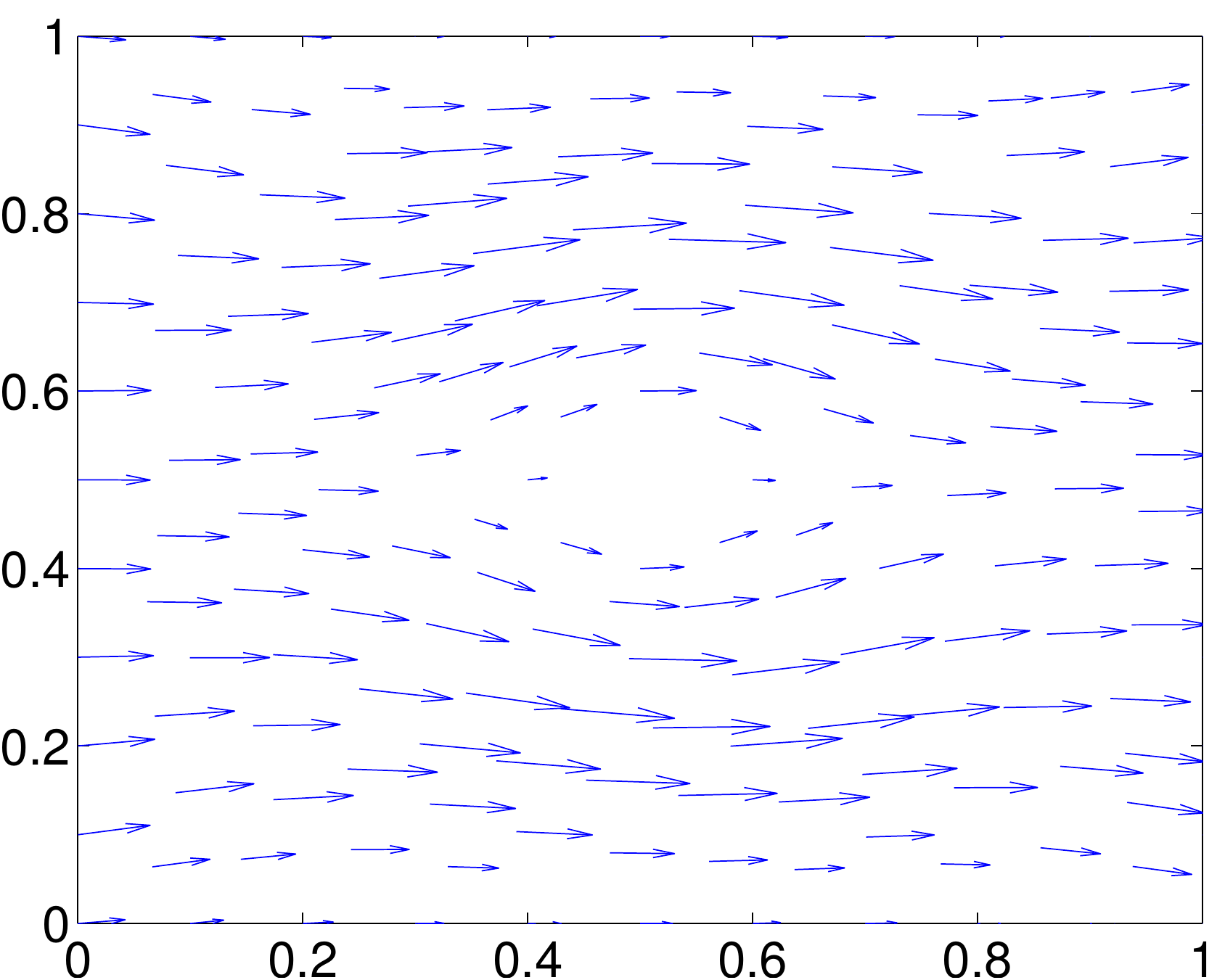}} \quad
  \resizebox{2.45in}{2.1in}{\includegraphics{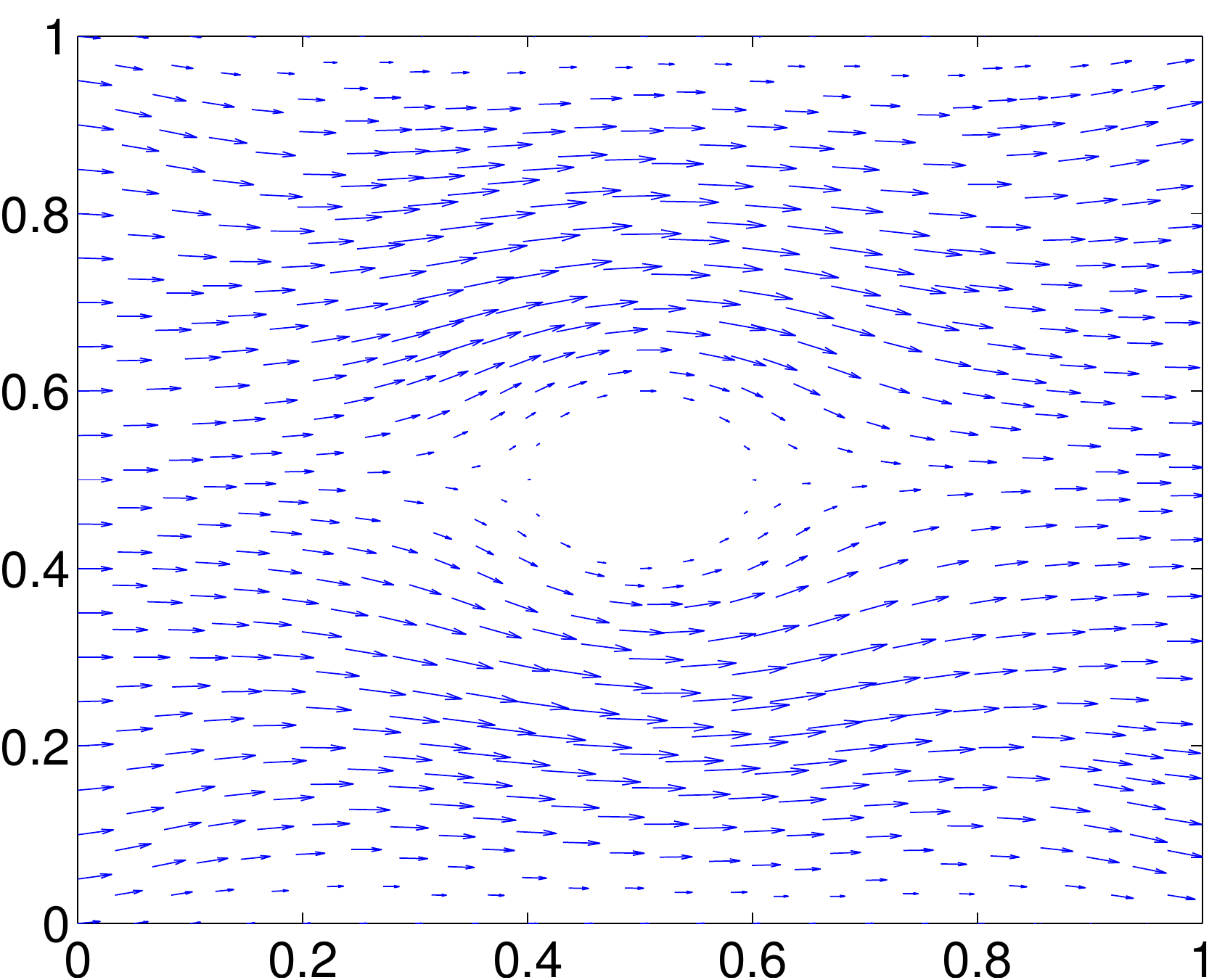}}
\end{tabular}
\caption{Example \ref{Num_ex4}: Vector fields of Velocity on Mesh Level 1 (Left) and Mesh Level 2 (Right)}
\label{Num_Ex4_1}
\end{figure}

\begin{figure}[!tb]
\centering
\begin{tabular}{c}
  \resizebox{2.7in}{2.45in}{\includegraphics{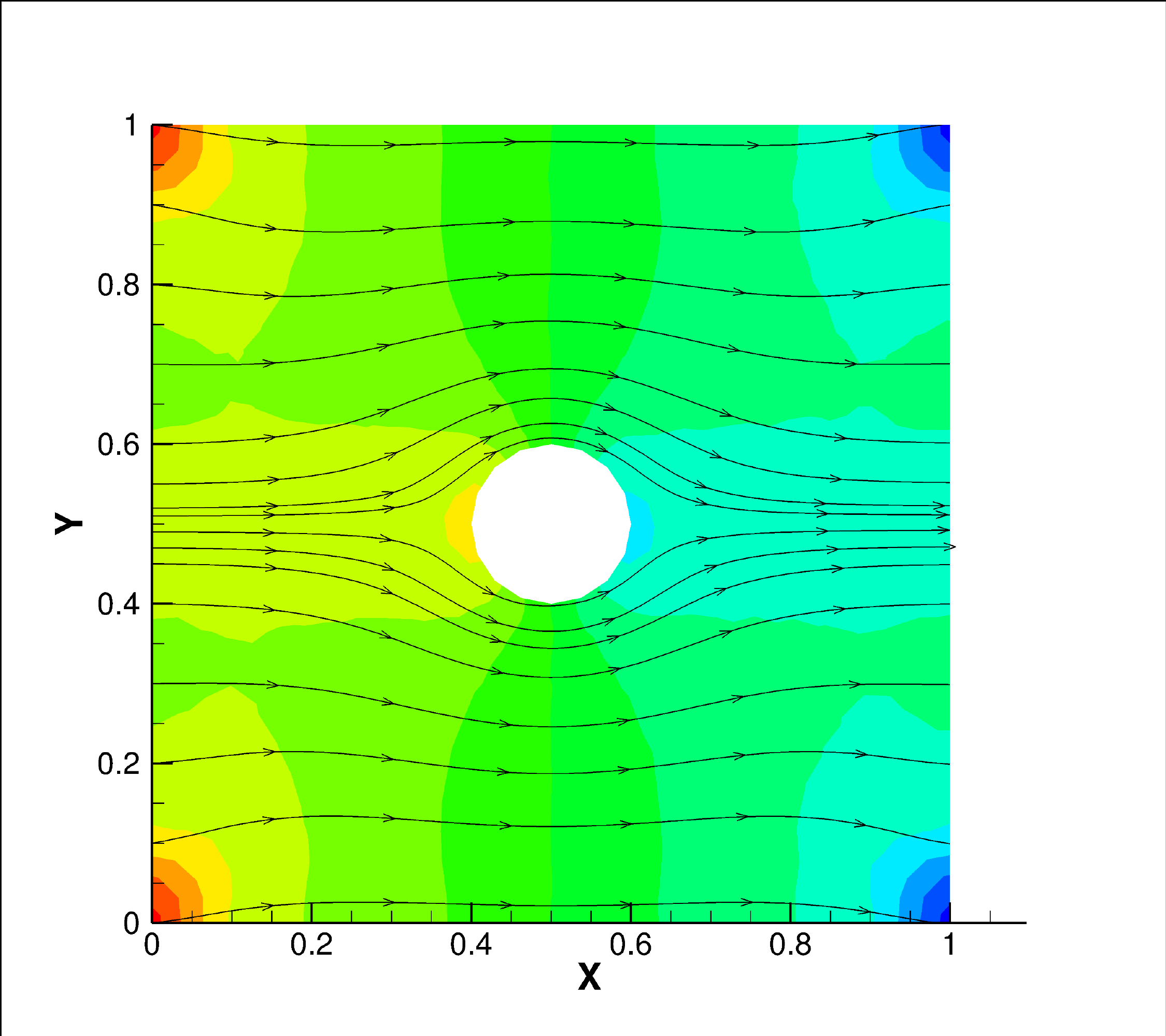}}
\end{tabular}
\caption{Example \ref{Num_ex4}: Streamlines of velocity and color contour of pressure.}
\label{Num_Ex4_2}
\end{figure}

\subsection{Example 5}\label{Num_ex5}
This example is used to test the backward facing step problem. This example is a benchmark problem. Let $\Omega=(-2,8)\times(-1,1)\backslash [-2,0]\times[-1,0]$, consider the Stokes problem with $\bbf=0,$ and Dirichlet boundary condition as:
\begin{eqnarray*}
\bu|_{\partial\Omega}={\bf g}=\begin{cases}
(-y(y-1)/10,0)^t,&\mbox{ if } x=-2;\\
(-(y+1)(y-1)/80,0)^t,&\mbox{ if } x=8;\\
(0,0)^t, &\mbox{else}.
\end{cases}
\end{eqnarray*}

\begin{figure}[!tb]
\centering
\begin{tabular}{c}
  \resizebox{4.45in}{1.1in}{\includegraphics{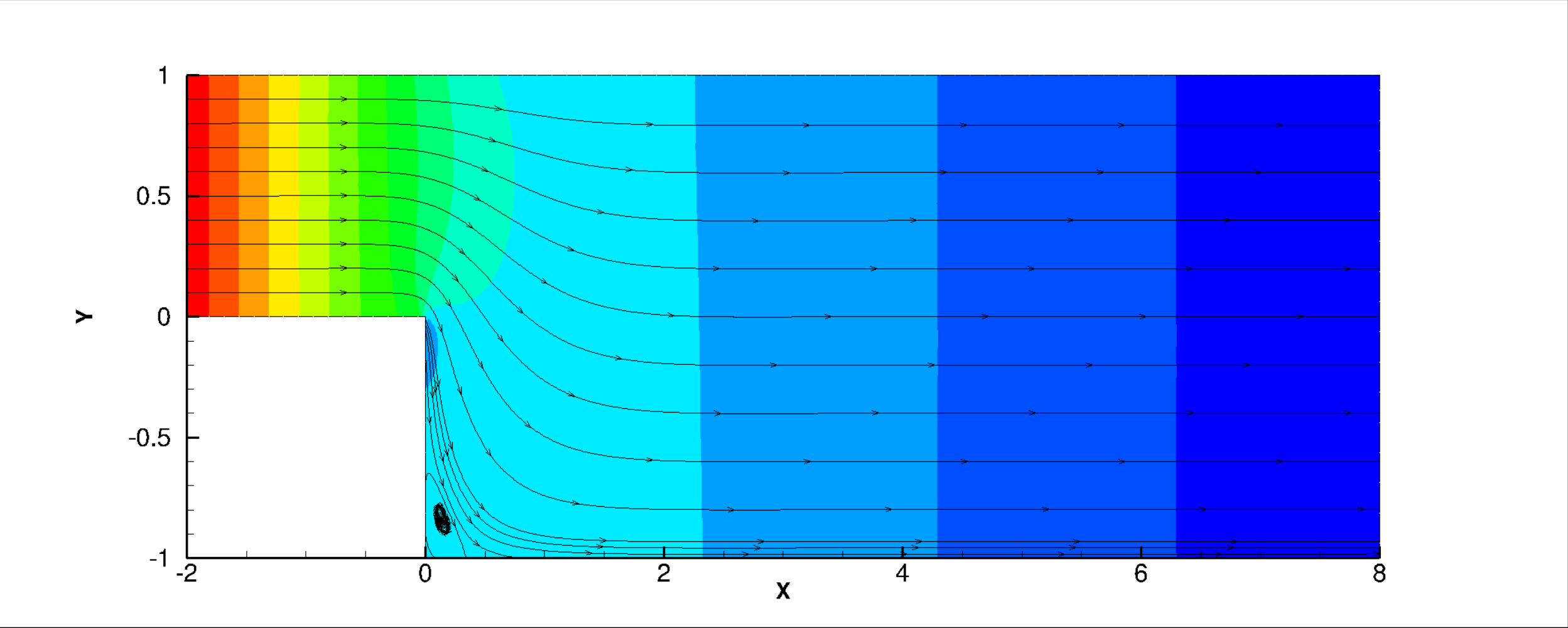}} \\
  \resizebox{4.45in}{1.1in}{\includegraphics{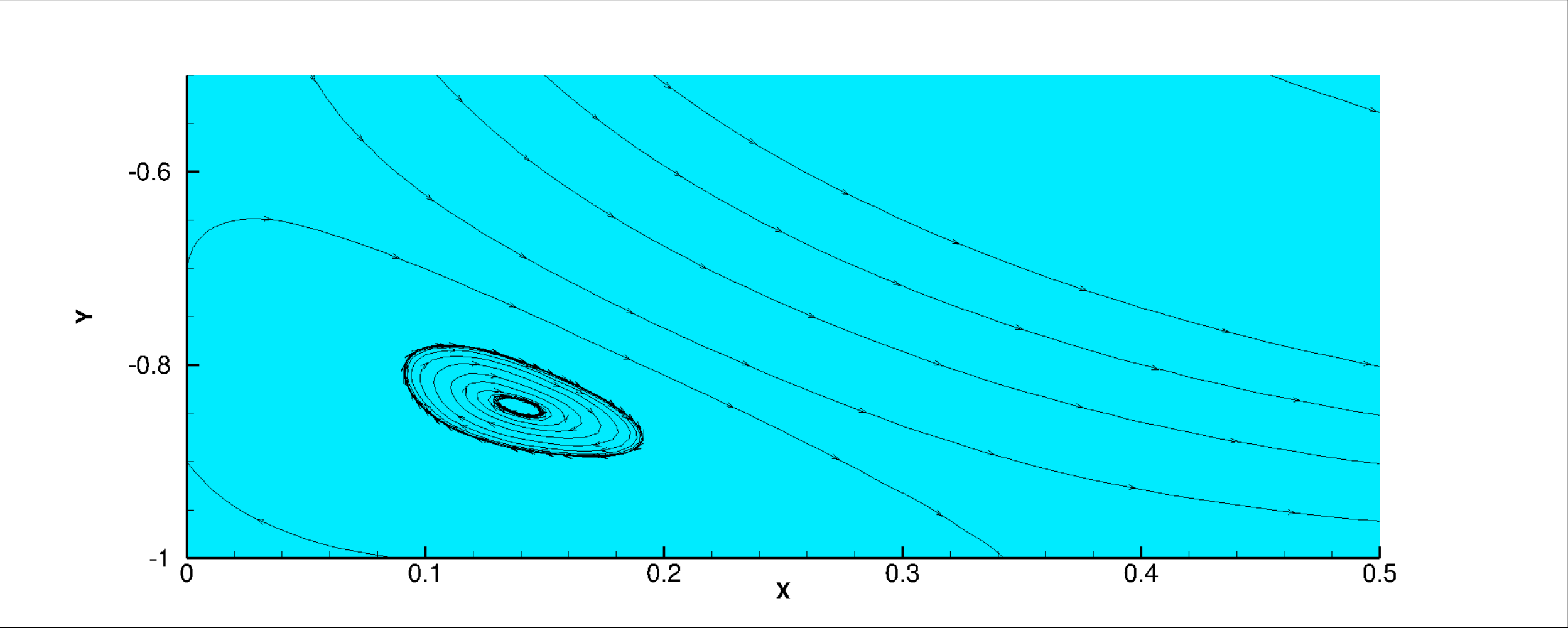}}
\end{tabular}
\caption{Example \ref{Num_ex5}: Streamlines of velocity and color contour of pressure (top); Zoom in plot.}
\label{Num_Ex5_1}
\end{figure}

Figure \ref{Num_Ex5_1} plots the streamlines of velocity and color contour of pressure. The plot shows that the pressure is high on the left and low on the right. A zoom figure of lower left corner $[0,0.5]\times[-1,-0.65]$ is plotted in Figure \ref{Num_Ex5_1}, which shows one eddy.

\subsection{Example 6}\label{Num_ex6}
The lid-driven cavity flow is considered in this example with $\Omega=(0,1)\times(0,1)$,  $\bbf=0$, and the Dirichlet boundary condition as:
\begin{eqnarray*}
\bu|_{\partial\Omega}={\bf g}=\begin{cases}
(1,0)^t,&\mbox{ if } y=1;\\
(0,0)^t, &\mbox{else}.
\end{cases}
\end{eqnarray*}

\begin{figure}[!tb]
\centering
\begin{tabular}{ccc}
  \resizebox{2.45in}{2.1in}{\includegraphics{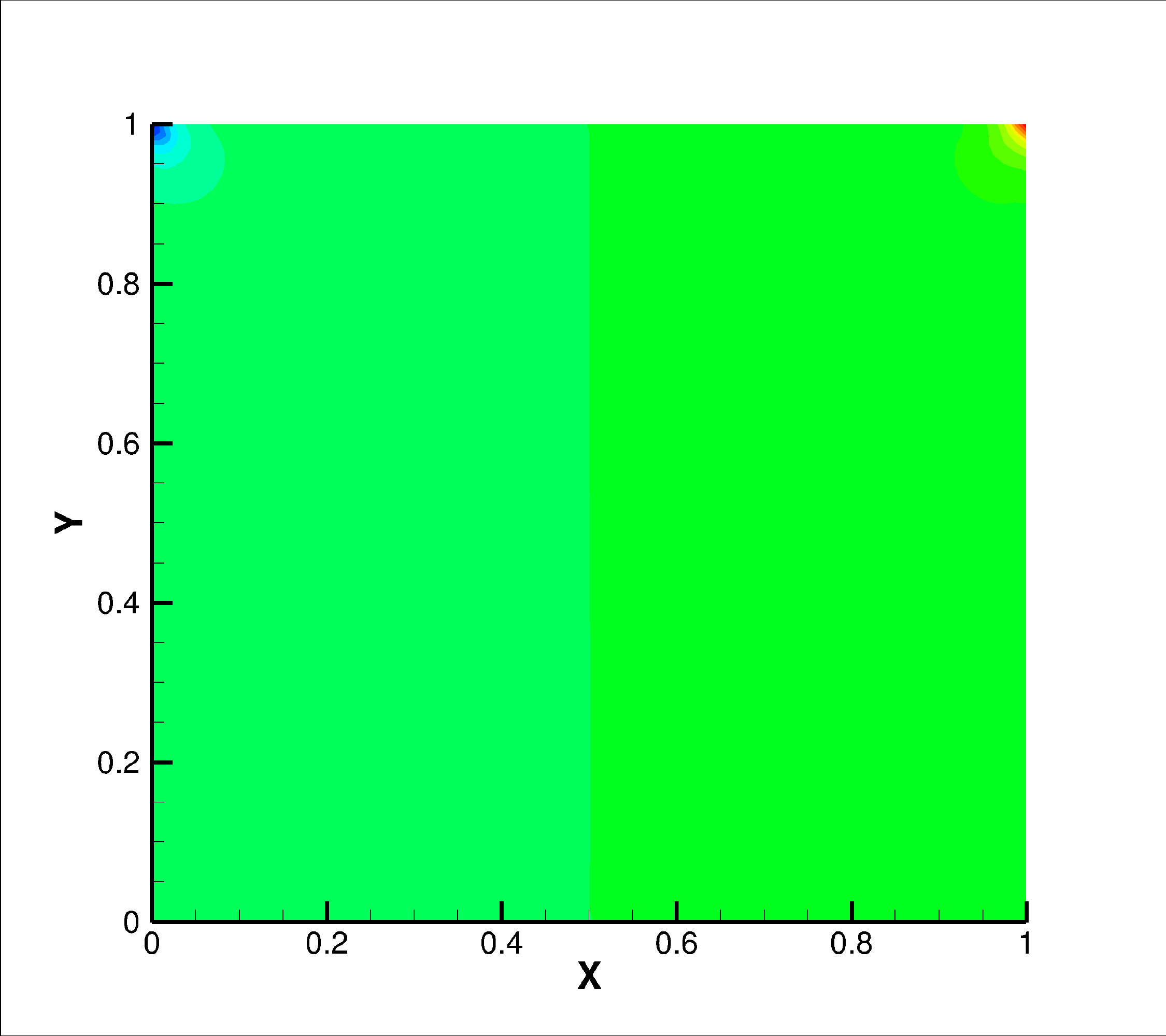}} \quad
  \resizebox{2.45in}{2.1in}{\includegraphics{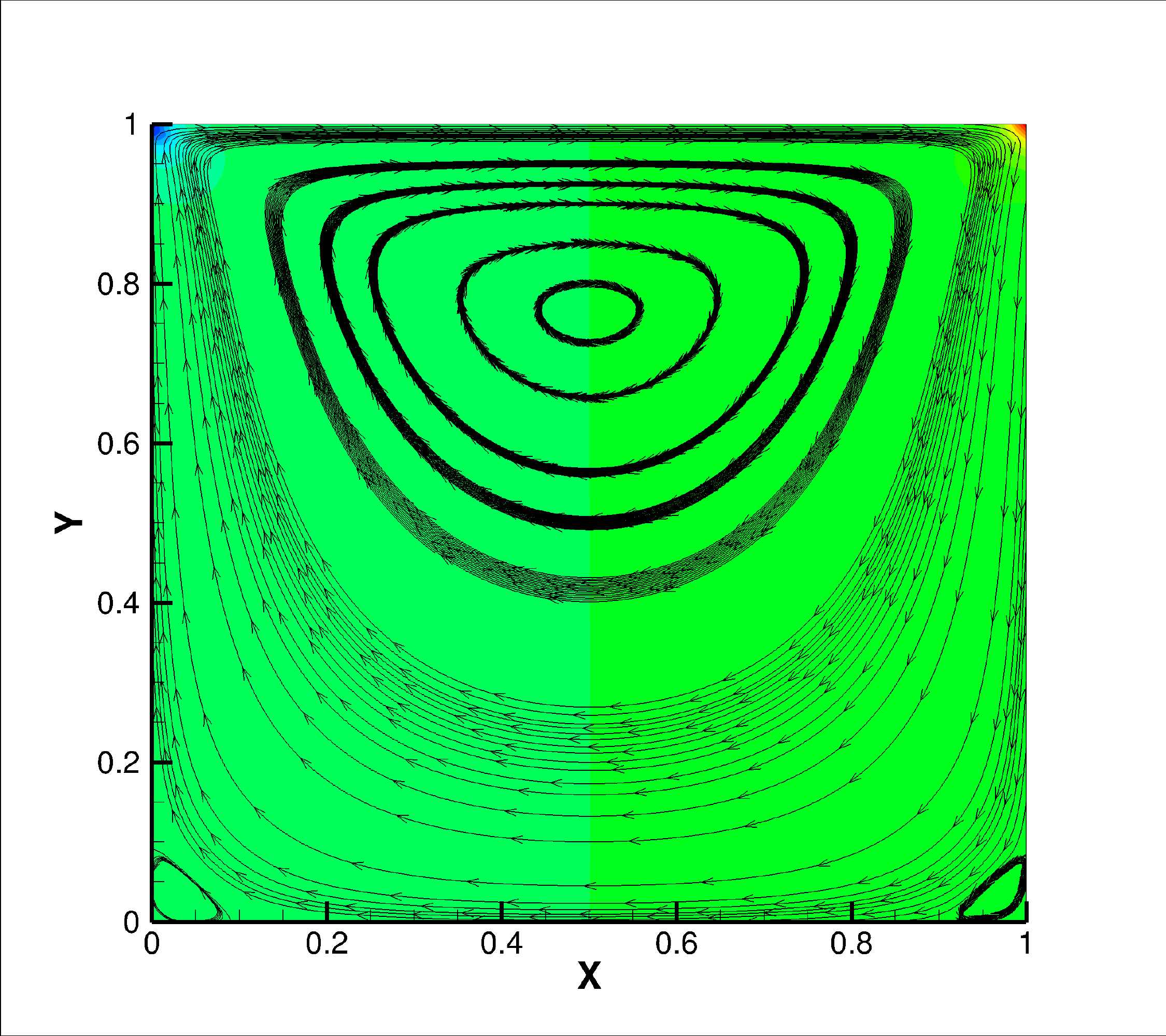}} \\
  \resizebox{2.45in}{2.1in}{\includegraphics{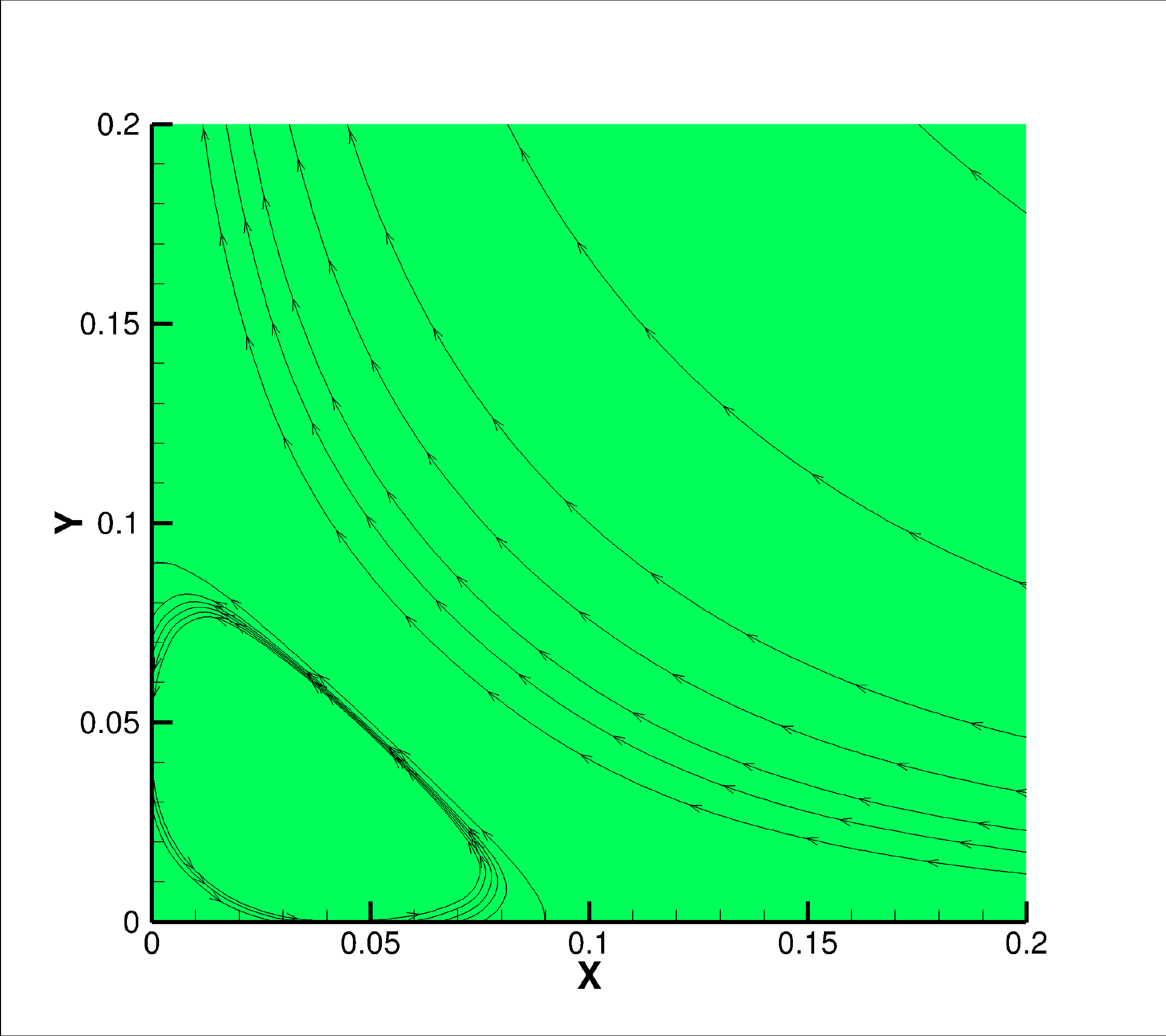}}\quad
  \resizebox{2.45in}{2.1in}{\includegraphics{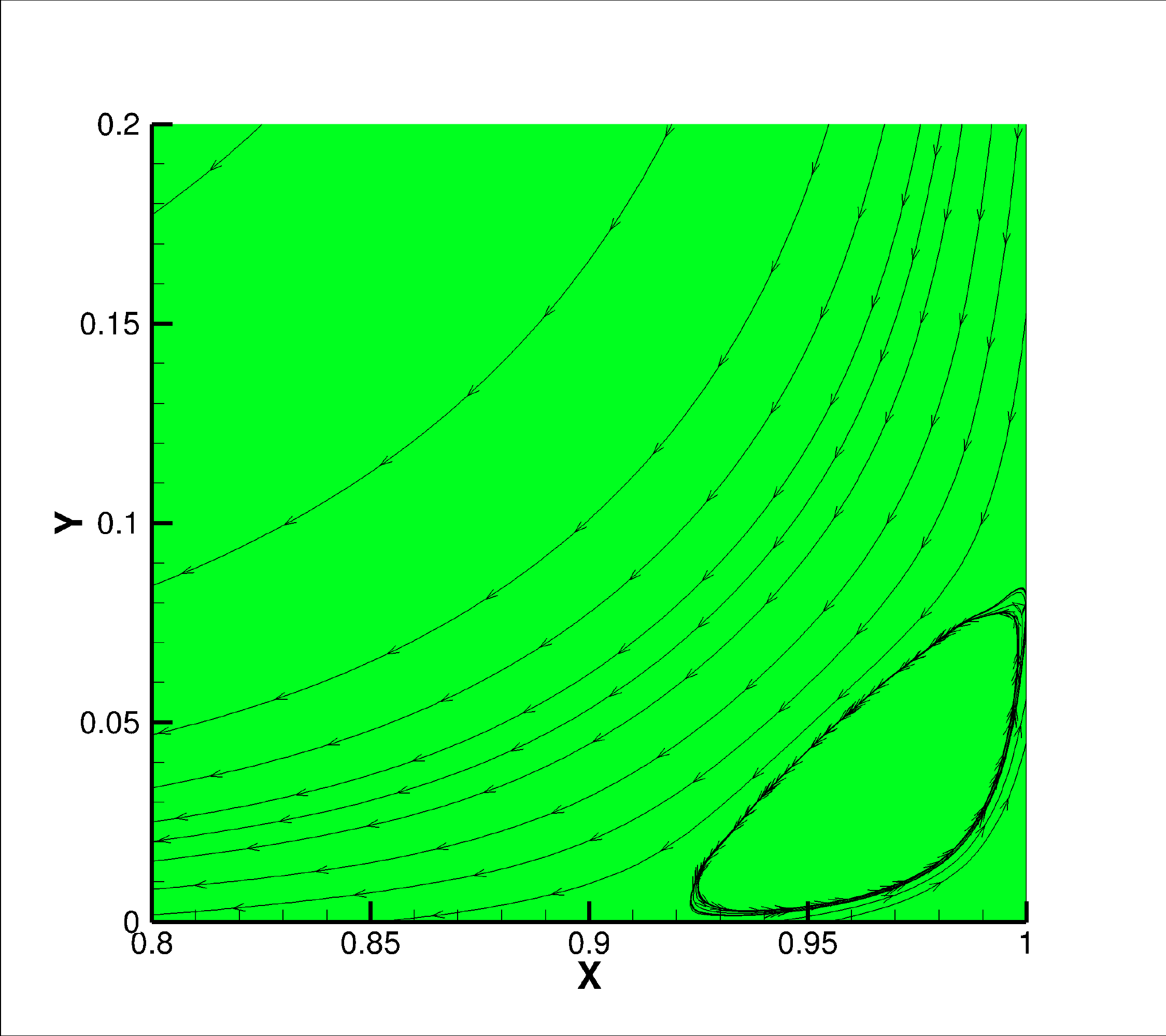}}
\end{tabular}
\caption{Example \ref{Num_ex6}: Color contour of pressure (top left); Streamlines of velocity (top right); Zoom in plot of two bottom corners.}\label{Num_ex6_1}
\end{figure}

Figure \ref{Num_ex6_1} displays the color contour of pressure $p_h$ and streamlines of velocity $\bu_h$ on a uniform mesh.  Two Moffat eddies at the bottom corners are detected in Figure \ref{Num_ex6_1}.

\end{document}